%% file: arXiv-TwoSpeedSojournTime.tex
\documentclass[]{article} 
\input{commands.tex}

\usepackage{amsthm}
\usepackage{geometry}
\theoremstyle{plain}
\newtheorem{theorem}{Theorem}[section]
\newtheorem{lemma}[theorem]{Lemma}
\newtheorem{proposition}[theorem]{Proposition}

\theoremstyle{definition}

\theoremstyle{remark}
\newtheorem*{remark}{Remark}

\begin{document}

\title{Sojourn time in a single server queue with threshold service rate control}

\author{
  Ivo Adan%
\footnote{
 Mechanical Engineering Department, Technische Universiteit Eindhoven,
 Postbus 513, 5600 MB EINDHOVEN,
 Tel.: +31-40-247 2932,
 {iadan@tue.nl}           
}
\and 
  Bernardo D'Auria%
\footnote{
 Statistics Department, Madrid University Carlos III, 
 Avda. Universidad, 30. 28911 Legan\'es (Madrid) Spain,
 Tel.: +34 916 248 804,
 {bernardo.dauria@uc3m.es}          
}
\thanks{The research was partially supported by the 
Spanish Ministry of Education and Science Grants
MTM2010-16519, 
SEJ2007-64500, 
and the Dutch Star grant of October 2013. 
The second author wants to thank the research institutes
ICMAT (Madrid, Spain) and EURANDOM (Eindhoven, The Netherlands)
for kindly hosting him during the development of this project.}
 }

\maketitle

\begin{abstract}
We study the sojourn time in a queueing system with a single exponential server, serving a Poisson stream of customers in order of arrival.
Service is provided at low or high rate, which can be adapted at exponential inspection times. When the number of customers in the system is above a given threshold, the service rate is upgraded to the high rate, and otherwise, it is downgraded to the low rate.
The state dependent changes in the service rate make the analysis of the sojourn time a challenging problem, since the sojourn time now also depends on future arrivals. We determine the Laplace transform of the stationary sojourn time and describe a procedure to compute all moments as well. First we analyze the special case of continuous inspection, where the service rate immediately changes once the threshold is crossed. Then we extend the analysis to random inspection times. This extension requires the development of a new methodological tool, that is \emph{matrix generating functions}. The power of this tool is that it can also be used to analyze generalizations to phase-type services and inspection times.
\end{abstract}


\begin{AMS}%
60K25,60K37,60D05.
\end{AMS}

\input{introduction.tex}
\input{model-immediate.tex}
\input{model-immediate-mom.tex}
\input{model-inspections.tex}
\input{Erlang-example.tex}
\input{anal-example.tex}
\input{model-inspections-mom.tex}
\input{num-experiments.tex}
\input{conclusion.tex}
\input{biblio.tex}

\input{appendix.tex}

\end{document}

%% file: commands.tex
\usepackage{amssymb} 
\usepackage{amsmath} 
\usepackage{tikz}
\usetikzlibrary{arrows,positioning,topaths,calc}
\usepackage[]{algorithm2e}
\newcount\mycount

\newcommand{\E}{\mathbb{E}}
\renewcommand{\Pr}{\mathbb{P}}
\newcommand{\tr}{{\top}}

\newcommand{\defeq}{:=}
\newcommand{\matL}{\Lambda}
\newcommand{\matM}{M}
\newcommand{\matG}{\Gamma}
\newcommand{\matGl}{\matG_0}
\newcommand{\matGr}{\matG_1}

\newcommand{\matT}{T}
\newcommand{\matI}{I}
\newcommand{\matH}{H}
\newcommand{\matS}{S}
\newcommand{\matA}{A}
\newcommand{\matB}{B}
\newcommand{\matZ}{Z}
\newcommand{\matU}{U}
\newcommand{\matY}{Y}

\makeatletter
\def\provideenvironment{\@star@or@long\provide@environment}
\def\provide@environment#1{%
        \@ifundefined{#1}%
                {\def\reserved@a{\newenvironment{#1}}}%
                {\def\reserved@a{\renewenvironment{dummy@environ}}}%
        \reserved@a
}
\def\dummy@environ{}
\makeatother
\provideenvironment{keywords}{\textbf{Key words. } }{}
\provideenvironment{AMS}{\textbf{AMS subject classifications. } }{}

\usepackage{cancel}

%% file: introduction.tex
\section{Introduction}
\label{intro}
We consider a single-server queueing system, where customers arrive according to a Poisson stream with rate $\lambda$ and receive service in order of arrival. The service requirements are exponential with mean $1$. The rate of the server can be either $\mu_0$, or $\mu_1$ and this service rate can be adapted at random inspection times that occur according to a Poisson stream with rate $\gamma$. 
For convenience, we think of $\mu_1$ as the fastest rate, i.e. $\mu_1 > \mu_0$, even if under the stability conditions this assumption may be removed. When the number of customers in the system is above the threshold $K$, the service rate is upgraded to the high rate $\mu_1$, and otherwise, it is downgraded to the low rate $\mu_0$. An important performance measure is the sojourn time. In this paper we aim to determine its stationary distribution. This is a challenging problem, since due to adaptable service rate, the sojourn time does not only depend on the state seen at arrival, but it also depends on future arrivals.

There is a considerable literature on the analysis of single server queueing systems with variable service rates, see e.g. \cite{asmussen,bekker1,cohen1,cohen2,gaver,harrison}. Those studies often assume that the service rates can be continuously adapted based on the queue content, and focus on the calculation of the steady-state workload distribution.
An exponential multi-server system is considered in \cite{mitrani}, with the feature that a reserved block of servers can be switched on (which takes an exponential switch-on time) when the number of customers in the system exceeds a certain threshold, and this block is immediately switched off when the number drops below another threshold. The emphasis in \cite{mitrani} is on the trade-off between the mean sojourn time and operating costs of the servers. In \cite{bekker2}, the stationary distribution of the workload is determined for an $M/G/1$ queue, where the service rate can not be continuously adapted, but only right after customer arrivals. In the literature, systems with adaptable service speed at inspection times have already been analyzed, we refer the reader to  \cite{bekker3,boxma} and the references therein.

The model with inspection rate $\gamma<\infty$ can be handled by considereing a two-stage birth-death process. This kind of model usually shows up in the analysis of retrial queues, where the state of the system has to keep track of the size of the retrial orbit. We refer the reader to the survey \cite{falin3}. In \cite{falin2}, the number of retrials of a generic customer is analyzed, which is a quantity directly related to the sojourn time and which depends on future arrivals to the system. Paper \cite{falin2} starts the analysis with a matrix equation that is similar to the one appearing in Section \ref{sc:inspection-times}, but it is able to reduce this equation to a scalar one by exploiting the fact that the retrial system has no buffer and only one server. Multi-channel systems are much more complicated to analyze and very little results are available about the sojourn time. Generally, what makes retrial systems more complicated than the one analyzed here, is the property that the rate at which retrial customers arrive to the system is proportional to the size of the orbit. This phenomenon is not appearing in our system, which is one of the reasons why our analysis is feasible.

As mentioned above, the focus in the current paper is on the sojourn time, not on the workload or number of customers in the system.
In Section \ref{sc:immediate-switching}, we first consider the special case of continuous inspection (so $\gamma = \infty$), where the service rate immediately changes once the threshold $K$ is crossed. This assumption simplifies the model, though it still contains the complication that the sojourn time depends on future arrivals. For the case of continuous inspection, we determine the Laplace transform of the stationary sojourn time and describe a procedure to compute all moments as well. The computation of the Laplace transform requires a recursive scheme and for the case $\gamma<\infty$ 
the Laplace transform can be expressed in terms of matrix functions that can be computed as solutions of a linear matrix system.

Then, in Section \ref{sc:inspection-times}, we proceed by extending the analysis to random inspection times occurring according a Poisson stream with rate $\gamma < \infty$. This extension, however, is not straightforward, but it requires the development of a new methodological tool, that is \emph{matrix generating functions}. By employing this tool we are able to find an expression for the Laplace transform of the stationary sojourn time, involving finitely many terms which can be recursively calculated. The analytical results are illustrated by numerical examples.

%% file: model-immediate.tex
\section{Model with continuous inspection}\label{sc:immediate-switching}

In this section we first consider the special case of continuous inspection, so $\gamma=\infty$. This implies that whenever the number of customers in the system exceeds the threshold $K > 0$, the rate of the server is immediately upgraded from the low rate $\mu_0$ to the high rate $\mu_1 > \mu_0$. As soon as the number of customers in the system becomes less or equal to $K$, the rate of the server is reduced to the low rate $\mu_0$ again.

Denoting by $Q(t)$, the number of customers in the system at time $t>0$, we have that the process is a continuous time Markov chain, the transition diagram of which is depicted in Figure \ref{fig:trans.diag.immediate}.
\input{pic-model-1.tex}

Denoting by $Q^*$ the stationary number of customers in the system, we have that its distribution is given by
\begin{equation}\label{eq:imm.stat.dist}
\pi_n
= \Pr(Q^* = n)
= \left\{
\begin{array}{rl}
(\lambda/\mu_0)^n \, \pi_0 & \quad \mbox{for } n \leq K \\
(\mu_1/\mu_0)^K \,(\lambda/\mu_1)^n \, \pi_0 & \quad \mbox{for } n > K
\end{array}
\right. \ ,
\end{equation}
and under the stability assumption $\mu_1>\lambda$, the value of $\pi_0$ is given by
\begin{equation}\label{eq:pi_0}
\pi_0
= \left(
\sum_{n=0}^K (\lambda/\mu_0)^n +
\frac{\lambda}{\mu_1-\lambda} \, (\lambda/\mu_0)^K
\right)^{-1} \, .
\end{equation}

We aim to compute the distribution of the sojourn time of a typical customer that arrives to the system in stationary regime. Note that, in order to do this, 
we can not use Little's distributional law \cite{keilson}, since future arrivals may affect the sojourn times of the customers already present in the system by inducing a change in the service rate.

As shown in Figure \ref{fig:2}, we identify a \emph{tagged} customer in the queue by a pair of numbers $(n,m)$, where $n$ stands for the position of the tagged customer in the queue, and where $m$ denotes the number of customers behind him. 
We denote the sojourn time of this customer $(n,m)$ by $S(n,m)$. The stationary 
sojourn time is denoted by $S^*$.

\input{system.tex}

For the Laplace transforms $\psi(s) = \E[e^{-s \, S^*}]$ and
$\psi(n,m,s) = \E[e^{-s \, S(n,m)}]$, the following relation holds by virtue of PASTA~\cite{wolff},
\begin{equation}\label{eq:Laplace.eq}
\psi(s) = \sum_{n=0}^\infty \psi(n+1,0,s) \, \pi_n \, .
\end{equation}
Hence, to compute the Laplace transform of the stationary sojourn time $S^*$, we need to compute the transforms $\psi(n,0,s)$ for each $n \geq 0$.

By using next-event analysis we have, for $n>0$,
\begin{equation}\label{eq:soj.rel}
S(n,m) =
\left\{\begin{array}{ll}
\frac{1}{\lambda+\mu_0} \, X + &
 \left\{\begin{array}{ll}
S(n-1,m) & \mbox{w.p. } \mu_0/(\mu_0+\lambda) \\
S(n,m+1) & \mbox{w.p. } \lambda/(\mu_0+\lambda)
\end{array}
\right.
\mbox{, as } n + m \leq K
\\ \\
\frac{1}{\lambda+\mu_1} \, X + &
 \left\{\begin{array}{ll}
S(n-1,m) & \mbox{w.p. } \mu_1/(\mu_1+\lambda) \\
S(n,m+1) & \mbox{w.p. } \lambda/(\mu_1+\lambda)
\end{array}
\right.
\mbox{, as } n + m > K
\end{array}
\right.
\end{equation}
where $X$ denotes an independent exponential random variable with rate $1$, and
$S(0,m) = 0$.
By Laplace transforming the relations \eqref{eq:soj.rel}, we get, for $n>0$,
\begin{equation}\label{eq:lap.rel}
\psi(n,m,s) =
 \frac{ \mu_{1\{n + m > K\}} \, \psi(n-1,m,s) + \lambda \,  \psi(n,m+1,s) }{\lambda+\mu_{1\{n + m > K\}}+s}
\end{equation}
with boundary conditions, $\psi(0,m,s) = 1$, for all $m \geq 0$, and where we used the indicator function $1\{A\} = 1$ if $A$ is true and $0$ otherwise. 

When $m \geq K$, it follows that for any $n >0$, 
the server will work at high speed during the whole sojourn time of the $(n,m)$-tagged customer. Hence $S(n,m)$ is Erlang distributed with parameters $n$ and $\mu_1$, and thus its Laplace transform is equal to
\begin{equation}\label{eq:only.H}
\psi(n,m,s) = (\mu_1/(\mu_1 + s))^n \quad \mbox{as } n > 0 \mbox{ and } m \geq K.
\end{equation}
The above equation is also valid for $n=0$.

Using expression \eqref{eq:only.H} in \eqref{eq:lap.rel}, the Laplace transforms $\psi(n,m,s)$ for $m < K$ can be recursively computed in $n$,
as the following lemma shows. The proof of the lemma is deferred to Appendix \ref{sec:appendix}.
\begin{lemma}\label{lm:rec.psi}
By defining
\begin{align*}
a_s(k) = & \mu_0/(\lambda+\mu_0+s) 1\{k \leq K\} +  \mu_1/(\lambda+\mu_1+s) 1\{k > K\} \, ; \\
b_s(k) = & \lambda/(\lambda+\mu_0+s) 1\{k \leq K\} + \lambda/(\lambda+\mu_1+s) 1\{k > K\} \, ,
\end{align*}
and $B_s(k,0)=1$ and $B_s(k,h+1) = B_s(k,h) \, b_s(k+h)$ for $k, h \ge 0$, we have 
\begin{align}\label{eq:rec.psi}
\psi(n,m,s)
= & B_s(n+m,K-m) \, \left(\frac{\mu_1}{\mu_1 + s}\right)^n  \nonumber \\
  &  + \sum_{k=m}^{K-1} a_s(n+k) \, B_s(n+m,k-m) \, \psi(n-1,k,s) \, ,
\end{align}
for $n >0$ and $0 \le m < K$. 
\end{lemma}
\begin{remark}
It can be easily shown that the value of $B_s(k,h)$ can be explicitly computed by the following formula
\begin{equation}\label{eq:B_s}
B_s(k,h)= \left(\frac{\lambda}{s+\lambda +\mu_1}\right)^h
\left(\frac{s+\lambda +\mu_1}{s+\lambda +\mu_0}\right)^{h \wedge (K-k+1)^+}
\end{equation}
with $a \wedge b = \min\{a,b\}$ and $(a)^+ = \max\{a,0\}$.
\end{remark}

Relation \eqref{eq:only.H} and Lemma \ref{lm:rec.psi} allow us to compute $\psi(n,m,s)$ for any $m,n\geq0$.
However, to calculate $\psi(s)$ in \eqref{eq:Laplace.eq} we still need to compute an infinite number of terms.
To overcome this issue we take advantage of the fact that, above the threshold $K$, the transition diagram is invariant towards the right, similarly to the standard $M/M/1$ queue.
To use this invariant property we introduce the following marginal $z$-transform
\begin{equation}\label{eq:phi.def}
  \phi(z,m,s) = \sum_{h=0}^\infty \psi(K+h+1,m,s) \, z^h \, ,
\end{equation}
valid for $|z|<1$.
In the following we show how to compute, in finitely many steps, the function $\phi(z,m,s)$. We use it to calculate the infinite sum in \eqref{eq:Laplace.eq} and then obtain a formula to compute the Laplace transform of the sojourn time as given in Proposition \ref{prop:Lapl.trans}.

By writing \eqref{eq:lap.rel} for $n=K+h+1$, multiplying by $z^h$ and summing over all $h \geq 0$, the following recursive equation holds
\begin{equation}\label{eq:phi.rel}
\phi(z,m,s)  =
 \frac{\mu_1 \, \psi(K,m,s)}{\lambda+\mu_1(1-z)+s}
+ \frac{\lambda \, \phi(z,m+1,s)}{\lambda+\mu_1(1-z)+s}  \ .
\end{equation}
As boundary value we have
\begin{align}\label{eq:phi(z,K,s)}
\phi(z,K,s)
= & \sum_{h=0}^\infty \left(\frac{\mu_1}{\mu_1 + s}\right)^{K+h+1} z^h
= \left(\frac{\mu_1}{\mu_1 + s}\right)^{K+1} \,
\sum_{h=0}^\infty \left(\frac{\mu_1 \, z}{\mu_1 + s}\right)^h \nonumber \\
= & \left(\frac{\mu_1}{\mu_1 + s}\right)^{K+1} \, \frac{\mu_1 + s}{\mu_1 (1-z)+ s} \ ,
\end{align}
from which the values of $\phi(z,m,s)$ can be recursively computed for $m=K-1,\ldots,0$, yielding
\begin{align}\label{eq:phi(z,m,s)}
\phi(z,m,s)
= & \sum_{h=0}^{K-1-m}\frac{\mu_1\,\lambda^h}{(\lambda+\mu_1(1-z)+s)^{h+1}} \psi(K,m+h,s) \nonumber \\
 & +  \left(\frac{\lambda}{\lambda+\mu_1(1-z)+s}\right)^{K-m} \phi(z,K,s) \, .
\end{align}
In particular we can compute, in finitely many steps, the value of $\phi(z,0,s)$.

Knowing the value of $\phi(z,0,s)$, expression \eqref{eq:Laplace.eq} can be finally computed as summarized in the following proposition.

\begin{proposition}\label{prop:Lapl.trans}
The Laplace transform of $S^*$ can be computed in the form
\begin{align}\label{eq:Laplace.eq.final}
\psi(s)
= & \pi_0 \, \sum_{h=0}^{K-1} \left[ \left(\frac{\lambda}{\mu_0}\right)^h \,\psi(h+1,0,s)
+  \left(\frac{\lambda}{\mu_0}\right)^K
\left(\frac{\lambda}{\mu_1}\right)^h \left(\frac{\mu_1}{\mu_1+s}\right)^{h+1} \psi(K,h,s) \right]\nonumber \\
 & + \pi_0  \left(\frac{\lambda }{\mu_0}\right)^K \left(\frac{\lambda }{\mu_1}\right)^K 
\frac{\mu_1}{\mu_1-\lambda} \left( \frac{\mu_1}{\mu_1+s}\right)^{2K}    \frac{\mu_1-\lambda}{\mu_1-\lambda +s}
\ .
\end{align}
\end{proposition}
\begin{proof}
The result follows from \eqref{eq:Laplace.eq} by splitting the sum in a finite, $n < K$, and infinite part,
\begin{align}\label{eq:Laplace.eq.alt}
\psi(s)
= & \pi_0 \, \sum_{n=0}^{K-1}  \left(\frac{\lambda}{\mu_0}\right)^n \,\psi(n+1,0,s)
+ \pi_0  \, \left(\frac{\lambda}{\mu_0}\right)^K \, \phi(\lambda/\mu_1, 0, s) \ .
\end{align}
For the last term we use \eqref{eq:phi(z,m,s)} and \eqref{eq:phi(z,K,s)} to get
\begin{align}\label{eq:phi.at.rho_1}
\phi(\lambda/\mu_1,m,s)
= & \sum_{h=0}^{K-1-m}
  \left(\frac{\lambda}{\mu_1}\right)^h \left(\frac{\mu_1}{\mu_1+s}\right)^{h+1} \psi(K,m+h,s) \nonumber \\
 & + \frac{\mu_1}{\mu_1-\lambda +s} 
\left(\frac{\lambda }{\mu_1}\right)^{K-m} 
\left( \frac{\mu_1}{\mu_1+s}\right)^{2K-m}
\end{align}
and the result follows by rearranging terms.
\end{proof}

The terms appearing in equation \eqref{eq:Laplace.eq.final} have the following nice probabilistic interpretation.
\begin{itemize}
\item With probability $\pi_h = \pi_0 (\lambda/\mu_0)^h$, $h<K$, the tagged user enters a system with $h$ customers 
and experiences a sojourn time, the Laplace transform of which is $\psi(h+1,0,s)$.
\item With probability $\pi_{K+h} = \pi_0 (\lambda/\mu_0)^K (\lambda/\mu_1)^h$, $0<h<K$, he finds $K+h$ customers waiting.
We slightly modify the system and assume that the tagged customer overtakes $h+1$ customers and occupies position $K$ instead of the last one in the queue. 
In addition, the server first serves the last $h+1$ customers.
Since the speed of the server depends on the number of customers waiting and not on their specific order of service,
the first $h+1$ services will be at speed $\mu_1$ taking an Erlang time with parameters $h+1$ and $\mu_1$ to complete.
What is left is the service time of the tagged customer, the Laplace transform of which is $\psi(K,h,s)$.
 \item With probability $\pi_{\geq 2K} = \pi_0 (\lambda/\mu_0)^K (\lambda/\mu_1)^K (\mu_1/(\mu_1-\lambda))$
the tagged customer finds at least $2K$ customers waiting.
As before, he is going to occupy position $K$, the sojourn time of which is Erlang distributed with parameters $K$ and $\mu_1$.
The number of customers he has overtaken is at least $K$, and the time it takes to complete their services is the sum of $K$ exponential
random variables with parameter $\mu_1$ plus a generic sojourn time of an $M/M/1$ queue having $\mu_1$ as service speed.
This last quantity is exponentially distributed with parameter $\mu_1-\lambda$.
\end{itemize}

\begin{remark}\label{rm:inv.lap.scalar}
From the Laplace transform of the stationary sojourn time given in \eqref{eq:Laplace.eq}
an explicit expression for the distribution can be obtained.
Indeed, the inverse transformation is straightforward as the Laplace transform is a rational polynomial, the poles of which are all located on the real axis.
To be more precise, the locations of the poles belong to the set
$$\mathcal A = \{ -(\lambda+\mu_1), \, -(\lambda+\mu_0), \, -\mu_1, \, -(\mu_1-\lambda)\} \, ,$$
implying that the density function is given by a linear combination of terms
$t^{k} e^{a \, t}$ for $a \in \mathcal A$ and $k=0,1,\ldots,\mbox{mult}(a)-1$,
where $\mbox{mult}(a)$ denotes the multiplicity of pole $a$.
\end{remark}

\begin{remark}
If we let $K\to\infty$ in \eqref{eq:Laplace.eq.final}, we recover  \eqref{eq:Laplace.eq}.
For any $n\geq0$, $\psi(n+1,0,s)$ becomes the Laplace transform of an Erlang distribution with parameters $n+1$ and $\mu_0$,
and $\psi(s)$ reduces to the Laplace transform of an exponential distribution with parameter $\mu_0-\lambda$,
that is the distribution of the sojourn time of a classical $M/M/1$ queue with service rate $\mu_0$.
\end{remark}
\begin{remark}
If $K=0$, only the last term in \eqref{eq:Laplace.eq.final} is different from zero.
Substituting $\pi_0=(\mu_1-\lambda)/\mu_1$, given in \eqref{eq:pi_0}, we get that 
$\psi(s)$ is the Laplace transform of an exponential distribution with parameter $\mu_1-\lambda$,
that is the distribution of the sojourn time of a classical $M/M/1$ queue with service rate $\mu_1$.
\end{remark}

%% file: pic-model-1.tex
\begin{figure*}[h]\centering
\begin{tikzpicture}[->,shorten >=1pt,node distance=.75cm,
state/.style={circle,draw,minimum size=0.85cm}]
\tiny
\node[state] (0) [] {$0$};
\node[state] (1) [right=of 0] {$1$};
\node[state,draw=none] (2) [right=of 1] {$\ldots$};
\node[state] (K) [right=of 2] {$K$};
\node[state] (K+1) [right=of K] {$K+1$};
\node[state] (K+2) [right=of K+1] {$K+2$};
\node[state,draw=none] (K+3) [right=of K+2] {$\ldots$};
\path [every node/.style={font=\tiny}]
	(0)  edge [bend left] node [above] {$\lambda$} (1);
\path [every node/.style={font=\tiny}]
	(1)	 edge [bend left] node [above] {$\lambda$} (2)
	     edge [bend left] node [below] {$\mu_0$} (0)
	(2)	 edge [bend left] node [above] {$\lambda$} (K)
	     edge [bend left] node [below] {$\mu_0$} (1)
	(K)    edge [bend left] node [above] {$\lambda$}(K+1)
	       edge [bend left] node [below] {$\mu_0$} (2)
	(K+1)  edge [bend left] node [above] {$\lambda$}(K+2)
	     edge [bend left] node [below] {$\mu_1$} (K)
	(K+2)  edge [bend left] node [above] {$\lambda$}(K+3)
	       edge [bend left] node [below] {$\mu_1$} (K+1)
	(K+3)  edge [bend left] node [below] {$\mu_1$} (K+2);
\end{tikzpicture}
\caption{Transition diagram for continuous ispection model}
\label{fig:trans.diag.immediate}       
\end{figure*}
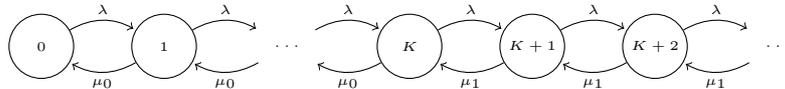

%% file: system.tex
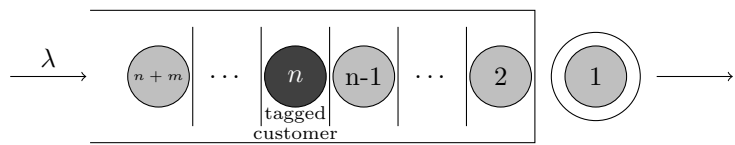
\begin{figure*}[h]\centering
\begin{tikzpicture}[%
cust/.style={fill=lightgray},
tCust/.style={fill=darkgray, text=white},
buff/.style={double,very thick,dashed},
incBuff/.style={thick,dashed},
server/.style={}]
\def\queueSlot{*.9}
\def\serverSize{.65\queueSlot}
\def\refSlot{.6*\serverSize}
\def\custSize{*.45\queueSlot}
\def\custText{0\custSize}
\def\queueSizeY{1.5*\serverSize}
\draw [->] (1,0) ++(-1.5\queueSlot-\refSlot,0) -- +(1,0) node [above, pos=0.5] {$\lambda$}; 
\def\mu_1mSlot{6.5}
\coordinate (Q1) at (7,0);
\draw [server] (Q1) circle (\serverSize);
\draw [cust] (Q1) circle (1\custSize)
		     ++(0,\custText) node {1};  
\draw [cust] (Q1) ++(-1\queueSlot-\refSlot,0) circle (1\custSize)
		     ++(0,\custText) node {2}; 
\draw (Q1) ++(-1.5\queueSlot-\refSlot,0.75*\queueSizeY) -- +(0,-1.5*\queueSizeY); 
\draw [cust] (Q1) ++(-2\queueSlot-\refSlot,0) node {\ldots}; 
\draw (Q1) ++(-2.5\queueSlot-\refSlot,0.75*\queueSizeY) -- +(0,-1.5*\queueSizeY); 
\draw [cust] (Q1) ++(-3\queueSlot-\refSlot,0) circle (1\custSize)
		     ++(0,\custText) node {n-1}; 
\draw (Q1) ++(-3.5\queueSlot-\refSlot,0.75*\queueSizeY) -- +(0,-1.5*\queueSizeY); 
\coordinate (tCust) at ($(Q1)+(-4\queueSlot-\refSlot,0)$);
\draw [tCust] (Q1) ++(-4\queueSlot-\refSlot,0) circle (1\custSize) node {$n$}; 
\draw [font=\scriptsize] (tCust) +(0,-.50) node {tagged};
\draw [font=\scriptsize] (tCust) +(0,-.75) node {customer};
\draw (Q1) ++(-4.5\queueSlot-\refSlot,0.75*\queueSizeY) -- +(0,-1.5*\queueSizeY); 
\draw [cust] (Q1) ++(-5\queueSlot-\refSlot,0) node {\ldots}; 
\draw (Q1) ++(-5.5\queueSlot-\refSlot,0.75*\queueSizeY) -- +(0,-1.5*\queueSizeY); 
\draw [cust] (Q1) ++(-6\queueSlot-\refSlot,0) circle (1\custSize) node {\tiny$n+m$}; 
\draw (Q1) ++(-0.5\queueSlot-\refSlot-\mu_1mSlot\queueSlot,+\queueSizeY)
	   -- ++(\mu_1mSlot\queueSlot,0) -- ++(0,-2*\queueSizeY) -- ++(-\mu_1mSlot\queueSlot,0); 
\draw [->] (Q1) ++(0.5\queueSlot+\refSlot,0) -- +(1,0); 
\end{tikzpicture}
\caption{Tagged customer $(n,m)$ at position $n$ in the queue, with $m$ customers behind him.}
\label{fig:2}       
\end{figure*} 

%% file: model-immediate-mom.tex
\subsection{First moment calculation}
As mentioned in Remark \ref{rm:inv.lap.scalar}, it is possible to compute the distribution of the sojourn time, but it is easier to compute the moments by using the relation
$\E[S^k] = (-1)^k \psi^{(k)}(0+)$. In this section we show how to compute the first moment. However, by taking higher order derivatives of the Laplace transform, recursive expressions can be obtained to compute all moments.

Let $\nu = \E[S]$ and $\nu_{n,m}=\E[S(n,m)]$. 
With $n>0$, from \eqref{eq:only.H} we have for $m \geq K$,
$\nu_{n,m} = n/ \mu_1$,  and  using \eqref{eq:rec.psi},
for $0 \le m < K$,
\begin{align}\label{eq:rec.psi.nu}
\nu_{n,m}
= &  \frac{n}{\mu_1}  B_{0^+}(n+m,K-m) - B'_{0^+}(n+m,K-m)  \nonumber \\ & 
+ \sum_{k=m}^{K-1}  \Big( \nu_{n-1,k} \, a_{0^+}(n+k) \, B_{0^+}(n+m,k-m) 
\nonumber \\ & 
\quad \quad \quad \quad - a'_{0^+}(n+k) \, B_{0^+}(n+m,k-m) \nonumber \\ & 
\quad \quad  \quad \quad  - a_{0^+}(n+k) \, B'_{0^+}(n+m,k-m) \Big)  \ ,
\end{align}
where
\begin{align*}
a'_{0^+}(k) = & -\mu_0/(\lambda+\mu_0)^2 1\{k \leq K\} -  \mu_1/(\lambda+\mu_1)^2 1\{k > K\} \, ; \\
b'_{0^+}(k) = & -\lambda/(\lambda+\mu_0)^2 1\{k \leq K\} - \lambda/(\lambda+\mu_1)^2 1\{k > K\} \, ,
\end{align*}
and $B'_{0^+}(k,0)=0$ and $B'_{0^+}(k,h+1) = B'_{0^+}(k,h) \, b_0(k+h) + B_0(k,h) \, b'_{0^+}(k+h) $. 

The following algorithm shows how to recursively compute  
$\nu_{n,m}$ for  $0 \le m < K$:\newline
\begin{algorithm}[H]
\For{i=$1$ \emph{\KwTo} n}{
\For{j=$1$ \emph{\KwTo} K-m}{
compute $\nu_{i,K-j}$
 }
 }
 \caption{Computing $\nu_{n,m}$ for $n>0$ and $0 \le m < K$}
\end{algorithm}

Finally, by applying Proposition \ref{prop:Lapl.trans}, we get
\begin{align}\label{eq:Laplace.eq.nu}
\nu
= & \pi_0 \, \sum_{h=0}^{K-1} \left[ \left(\frac{\lambda}{\mu_0}\right)^h \,\nu_{h+1,0}
+  \frac{\lambda^K}{\mu_0^K} \left(\frac{\lambda}{\mu_1}\right)^h \left( \nu_{K,h} + \frac{h+1}{\mu_1} \right) 
 \right] \nonumber \\
 & + \pi_0 \left(\frac{\lambda }{\mu_0}\right)^K \left(\frac{\lambda }{\mu_1}\right)^K \left(\frac{2K}{(\mu_1-\lambda )}+\frac{\mu_1}{(\lambda -\mu_1)^2}\right)
\ .
\end{align}

%% file: model-inspections.tex
\section{Model with inspection times}\label{sc:inspection-times}
In this section we analyze the system where inspection times occur according to a Poisson stream with rate $\gamma < \infty$. So, in this case, there is no continuous inspection and adaptation of the service rate is delayed (with an exponential time) when the number of customers in the system crosses the threshold $K$. If at an inspection time
the system is found congested with more than $K$ customers,
the service rate is immediately set to the fast rate $\mu_1$
and otherwise, if at most $K$ customers are present,
the service rate is set to the low rate $\mu_0$.

Now we need to include the service rate in the state description of the system, resulting in the Markov chain shown in Figure \ref{fig:trans.diag.inspections}.
Note that for any number of customers in the system, the service rate can be high and low.
%

\input{pic-model-2.tex}

Denoting by $\mathcal{M}$ the stationary random service rate, let $\pi_{0n} = \Pr(\mathcal{M}=\mu_0,Q^*=n)$ and $\pi_{1n} = \Pr(\mathcal{M}=\mu_1,Q^*=n)$ be the stationary probabilities to find $n$ customers in the system with the server working at rate $\mu_0$ and $\mu_1$, respectively.
In what follows, the quantity $\pi_n$ denotes the column vector with components $(\pi_{0n} ,\pi_{1n})^\tr$, where $(\cdot)^\tr$ is the transposition operator.
The stationary distribution satisfies the balance equations
\begin{equation}\label{eq:insp.stat.dist.rel}
\def\ns{\kern-.75em}
\begin{array}{llllcr}
                           \ns &- \; H_1 \, \pi_0 \ns &+\;  \matM \, \pi_1        \ns &= 0 &\\
\matL \, \pi_{n-1} \ns &- \; H_2 \, \pi_n \ns &+\;  \matM \, \pi_{n+1} \ns &= 0 & \quad 1 \leq n \leq K \, ;\\
\matL \, \pi_{n-1} \ns &- \; H_3 \, \pi_n \ns &+\;  \matM \, \pi_{n+1} \ns &= 0 & \quad n > K \, ,
\end{array}
\end{equation}
where the transition matrices are defined by
\begin{equation*}
\matM =  \left(\begin{array}{cc} \mu_0 & 0 \\ 0 & \mu_1 \end{array}\right)\ , \quad
\matL =   \left(\begin{array}{cc} \lambda & 0 \\ 0 & \lambda \end{array}\right) \ ,
\end{equation*}
and $H_1 = \matL + \matG_2$,
$H_2 = \matM + \matL + \matG_2$
and $H_3 = \matM + \matL + \matG_3$, where
\begin{equation*}
\matG_2 =  \left(\begin{array}{cc} 0 & -\gamma \\ 0 & \gamma \end{array}\right) \ , \quad
\matG_3 =  \left(\begin{array}{cc} \gamma & 0 \\ -\gamma & 0 \end{array}\right)
\end{equation*}

From the theory on quasi-birth-death processes \cite{Neuts94,latouche}, we conclude that for $n>K$, the
stationary probability vector $\pi_n$ can be written in the form
\begin{equation}\label{eq:pi.K+h}
\pi_{K+h} = R^h \, \pi_K, \quad h \geq 0 \ ,
\end{equation}
where the matrix $R$ is the minimal non-negative solution of the matrix equation
\begin{equation}\label{eq:R.mat.eq}
\matL - H_3 \,R + \matM \,R^2 = 0 \ .
\end{equation}
Using the probabilistic interpretation of $R$, or by solving the matrix equation \eqref{eq:R.mat.eq}, it follows that
$R$ is of triangular form and in particular, it is equal to
\begin{equation}\label{eq:matrix.R}
R = \left( \begin{array}{cc}
R_{00} & 0 \\
\frac{\gamma}{\mu_1} \frac{R_{00}}{1-R_{00}} & \frac{\lambda}{\mu_1}
\end{array} \right) \ ,
\end{equation}
with $R_{00} = \frac{\mu_0 + \gamma +\lambda }{2 \mu_0}-\sqrt{\left(\frac{\mu_0+\gamma +\lambda }{2\mu_0}\right)^2-\frac{ \lambda  }{\mu_0}}$.

The value of $\pi_K$ can be computed by the normalizing equation
\begin{equation}\label{eq:norm.eq}
\sum_{k=0}^{K-1} e \, \pi_k  + e \, (I -R)^{-1} \pi_K = 1 \ ,
\end{equation}
with $e$ the all-one (row) vector.

By PASTA, as in \eqref{eq:Laplace.eq},
the Laplace transform of the stationary sojourn time
is given by
\begin{equation}\label{eq:Laplace.eq.inf.insp}
\psi(s) = \sum_{n=0}^\infty \psi(n+1,0,s) \, \pi_n \, ,
\end{equation}
where $\psi(n,m,s)$
denotes the row vector $(\psi_0(n,m,s),\psi_1(n,m,s))$, with
$\psi_i (n,m,s)$ being the Laplace transform of the sojourn time
$S_i (n,m)$ of a tagged customer who is at position $(n,m)$ and the service rate is $\mu_i$, $i = 0, 1$.

By using next-event analysis, we get the following recursive equations
for the sojourn times, $S_i(n,m)$, $i=0,1$, $n>0$
\begin{align}\label{eq:soj.rel.insp}
S_i(n,m) &=
\frac{X}{\lambda+\mu_i+\gamma} +
 \left\{\begin{array}{ll}
S_i(n-1,m) & \mbox{w.p. } \mu_i/(\lambda+\mu_i+\gamma)  \\
S_i(n,m+1) & \mbox{w.p. } \lambda/(\lambda+\mu_i+\gamma)  \\
S_{1\{n + m > K\}}(n,m)  & \mbox{w.p. } \gamma/(\lambda+\mu_i+\gamma)
\end{array}
\right.
\end{align}
where $X$ denotes an independent exponential random variable with rate $1$,
and $S_i (0,m) = 0$.
Taking the Laplace transform of \eqref{eq:soj.rel.insp} yields the equation
\begin{align}\label{eq:lap.rel.insp}
\psi(n,m,s) \, (H(s)-\matG_{1\{n+m>K\}})
= \psi(n-1,m,s) \, \matM + \psi(n,m+1,s) \,  \matL
\end{align}
for $n>0$,  where
\begin{equation*}
H(s) = (\gamma+s) \,  \matI + \matL + \matM \ , \quad
\matGl =  \left(\begin{array}{cc} \gamma & \gamma \\ 0 & 0 \end{array}\right) \ , \quad
\matGr =  \left(\begin{array}{cc} 0 & 0 \\ \gamma & \gamma \end{array}\right) \ ,
\end{equation*}
and $\psi (0,m,s) = e$, with $e$ the all-one (row) vector and $\matI$ the identity matrix.

Similarly to \eqref{eq:only.H},
when $m \geq K$ and for any $n >0$, we have that whenever an inspection occurs,
the service rate is set and kept at the value $\mu_1$ till the end of the sojourn time of the tagged customer.
This implies that $\psi(n,m+1,s)=\psi(n,m,s)$ for $m \ge K$ and $n>0$, and substitution in \eqref{eq:lap.rel.insp} yields
\begin{equation}\label{eq:only.H.insp}
\psi(n,m,s) = e \, \matT^n(s)
  \quad \mbox{as } n > 0 \mbox{ and } m \geq K \ ,
\end{equation}
with $\matT(s)= \matM \, (H(s)-\matGr-\matL)^{-1} $.

Equations \eqref{eq:only.H.insp} and \eqref{eq:lap.rel.insp}
allow us to get the values of $\psi(n,m,s)$ for any $n,m\geq0$.
However to compute expression \eqref{eq:Laplace.eq.inf.insp} in finitely many steps,
we still need to find a way to handle the infinite sum.
So far, the analysis proceeds as in Section \ref{sc:immediate-switching}, and thus the next step would be to
introduce the marginal $z$-transforms corresponding to \eqref{eq:phi.def},
that is, $\phi_i(z,m,s) = \sum_h \psi_i(K+h+1,m,s) \, z^h$.
However, this approach immediately fails, since the stationary probability distribution \eqref{eq:pi.K+h} calls for a matrix generalization.
The main contribution of this work is to provide this generalization
by the introduction of the following \emph{matrix generating function},
\begin{equation}\label{eq:phi.def.insp}
  \phi(\matZ,m,s) = \sum_{h=0}^\infty \psi(K+h+1,m,s) \, \matZ^h \ ,
\end{equation}
where $\matZ$ is any matrix with eigenvalues contained in the open unit disk of the complex plane.

\begin{remark}
Since the absolute value of the Laplace transform $\psi_i(n,m,s)$ is less or equal to one, the assumption on the eigenvalues of $\matZ$ implies that the matrix generating function $\phi(\matZ,m,s)$ is well defined.
\end{remark}

Let us rewrite expression \eqref{eq:lap.rel.insp} for $n > K$ in the alternative form,
\begin{align}\label{eq:lap.rel.insp.alt}
\psi(n,m,s)
= \psi(n-1,m,s) \, \matT_M(s) + \psi(n,m+1,s) \, \matT_{\matL}(s)
\end{align}
with $\matT_A(s) = \matA \, (H(s)-\matGr)^{-1} $, $\matA\in\{\matL,\matM\}$.
Multiplying expression \eqref{eq:lap.rel.insp.alt} on the right by $\matZ^h$, for $n = K+h+1$, and then
summing over $h \ge 0$ and using that
$\matT \, \matZ^h \,   \matT^{-1} = (\matT\, \matZ \, \matT^{-1}  )^h \ ,$
we get a recursive equation for $\phi(\matZ,m,s)$,
\begin{align}\label{eq:phi.rel.insp}
\phi(\matZ,m,s)= &
  \psi(K,m,s) \, \matT_M(s) +  \phi(\matT_M(s) \, \matZ \, \matT_M^{-1}(s),m,s) \, \matT_M(s) \, \matZ \nonumber \\
& +  \phi(\matT_{\matL}(s) \,\matZ \, \matT_{\matL}^{-1}(s),m+1,s) \, \matT_{\matL}(s) \ .
\end{align}
The main difference between equations \eqref{eq:phi.rel} and \eqref{eq:phi.rel.insp} is that in the latter we loose the commutative property of the product and the functions $\phi$ need to be evaluated for different values of their arguments.
The boundary condition is obtained from \eqref{eq:only.H.insp},
\begin{align}\label{eq:phi(z,K,s).insp}
\phi(Z,K,s)
= & e \, \matT^{K+1}(s)  \, \left(\sum_{h=0}^\infty  \matT^h(s) \, \matZ^h \, \right) 
= e \, \matT^{K+1}(s) \, \matS(\matZ,I,\matT(s)) \, ,
\end{align}
with $I$ being the identity matrix and where we employed the definition,
\begin{equation}\label{def:s}
\matS(\matZ,\matA,\matB) \defeq \sum_{h=0}^\infty \matB^h \, \matA \, \matZ^h \ .
\end{equation}
The matrix $\matS(\matZ,\matA,\matB)$ is well defined for any matrix $\matZ, \matA, \matB$ with
$\matZ$ and $\matB$ having all eigenvalues inside the closed and open disk, respectively (so that the series converges). Note that $\matT(s)$ in \eqref{eq:phi(z,K,s).insp} has all eigenvalues inside the open unit disk. The matrix $\matS(\matZ,\matA,\matB)$
can be computed as the solution of a matrix equation as shown in the following lemma. The proof is deferred to the appendix.
\begin{lemma}\label{lm:S.function}
Let $\matZ$, $\matA$ and $\matB$ be three matrices with $\matZ$ and $\matB$ having all eigenvalues in the closed and open disk, respectively, then the matrix function $\matS=\matS(\matZ,\matA,\matB)$ is the unique solution of the following matrix system,
\begin{equation}\label{eq:S.syst}
\matS - \matB \, \matS \,  \matZ = \matA \ .
\end{equation}
\end{lemma}

The next proposition shows that, in order to compute the Laplace transform of the stationary sojourn time $S^*$ in terms of a finite number of addends, only the value of $\phi(R,0,s)$ is needed.
\begin{proposition}\label{prop:lap.trans}
The Laplace transform of $S^*$ can be computed in the form
\begin{align}\label{eq:Laplace.eq.insp}
\psi(s)
= & \sum_{n=0}^{K-1} \psi(n+1,0,s) \, \pi_n + \phi(R,0,s) \, \pi_K \ .
\end{align}
\end{proposition}
\begin{proof}
The result follows from \eqref{eq:Laplace.eq.inf.insp} by splitting the sum in a finite, $n < K$, and infinite part.
For the latter part, we express $\pi_{K+h}$ as in \eqref{eq:pi.K+h} for $h\ge0$, and apply definition \eqref{eq:phi.def.insp}.
\end{proof}

The computation of $\phi(R,0,s)$ requires some additional machinery with respect to the one developed in Section \ref{sc:immediate-switching} for the scalar case.
Before giving the statement of the main result we need the following technical lemma, the proof of which is deferred to the appendix. The lemma states that the infinite sum of matrices appearing at the left-hand side of \eqref{eq:rel.S} can be recognized as a matrix function $\matS$, which can be computed from the matrix system \eqref{eq:S.syst}.
\begin{lemma}\label{lm:rel.S}
Let $\matZ$, $\matA$ and $\matB$ be three matrices with $\matZ$ and $\matB$ having all eigenvalues in the closed and open disk, respectively, and let $T_1$ and $T_2$ be invertible matrices with $T_1$ having all eigenvalues in the open disk, then the following relation holds,
\begin{equation}\label{eq:rel.S}
\sum_{h=0}^\infty 
 \matS(\matT_2 \, \matT_1^h \,\matZ \, \matT_1^{-h} \, \matT_2^{-1}, \matA, \matB) 
 \, \matT_2 \, \matT_1^h \, \matZ^h 
=
\matS(\matZ, \matA \, \matT_2 \, \matS(\matZ, I, \matT_1) , \matB) \, .
\end{equation}
\end{lemma}
The following result allows us to compute the value of $\phi(R,m,s)$ in finitely many steps.

\begin{theorem}\label{th:phi.rec.insp}
The values of $\phi(\matZ,m,s)$ for $0 \leq m \leq K$ can be computed by the following equation
\begin{align}\label{eq:phi.rec.insp}
\phi(\matZ,m,s)
= &  \sum_{k=m}^{K-1} \psi(K,k,s) \, \matT_M(s) \, \matU_M(\matZ,k-m,s) \nonumber \\
& + \psi(K,K+1,s) \, \matT(s) \, \matU(\matZ,K-m,s)  \ ,
\end{align}
where the matrices $\matU_M(\matZ,k,s)$ and $\matU(\matZ,k,s)$ are defined as
\begin{eqnarray*}
\matU_M(\matZ,k,s) & = &
\matS(\matZ,(\matT_{\matL}(s) \, \matS(\matZ,I,\matT_M(s)))^k,\matT_M(s)) \ ,
\\
\matU(\matZ,k,s) & = &
\matS(\matZ,(\matT_{\matL}(s) \, \matS(\matZ, I, \matT_M(s)) )^k,\matT(s)) \ .
\end{eqnarray*}
\end{theorem}
\begin{proof}
Using \eqref{eq:phi(z,K,s).insp} and \eqref {eq:only.H.insp}, it follows that equation \eqref{eq:phi.rec.insp} holds for $m=K$, 
where it is assumed that the value of the sum is zero.
We prove by induction that it also holds for all $m<K$.
We first derive a recursive equation satisfied by $\phi(\cdot,m,s)$ in terms of $\phi(\cdot,m+1,s)$.

By substituting
$\matT_M(s) \, \matZ \, \matT_M^{-1}(s)$ for $Z$ in \eqref{eq:phi.rel.insp} we get an expression for 
$\phi(\matT_M(s) \, \matZ \, \matT_M^{-1}(s),m,s)$, and subsequently substituting this expression in the right-hand side of  \eqref{eq:phi.rel.insp}, yields
\begin{align}\label{eq:phi.rec.1.insp}
\phi(\matZ,m,s)= &
  \psi(K,m,s) \, \matT_M(s) +  \psi(K,m,s)  \, \matT_M^2(s)  \, \matZ \nonumber \\
& + \phi(\matT_M^2(s) \, \matZ \, \matT_M^{-2}(s),m,s) \, \matT_M^2(s) \,  \matZ^2 \nonumber \\
& + \phi(\matT_{\matL}(s) \,\matT_M(s) \, \matZ \, \matT_M^{-1}(s) \, \matT_{\matL}^{-1}(s),m+1,s)
\, \matT_{\matL}(s)  \, \matT_M(s) \,  \matZ \nonumber \\
& + \phi(\matT_{\matL}(s) \,\matZ \, \matT_{\matL}^{-1}(s),m+1,s) \, \matT_{\matL}(s) 
\end{align}
and iterating this equation leads to
\begin{align*}
\phi(\matZ,m,s) = 
& \psi(K,m,s)   \,\matT_M(s) \, \left(\sum_{h=0}^\infty \matT_M^h(s) \, \matZ^h \right) \nonumber \\
& + \sum_{h=0}^\infty \phi(\matT_{\matL}(s) \,\matT_M^h(s) \, \matZ \, \matT_M^{-h}(s)\, \matT_{\matL}^{-1}(s),m+1,s) 
\, \matT_{\matL}(s)  \, \matT_M^h(s) \, \matZ^h \ ,
\end{align*}
which can be rewritten as
\begin{align}\label{eq:phi.rec.inf.insp}
\phi(\matZ,m,s)
= & \psi(K,m,s) \,\matT_M(s) \, \matS(\matZ,I,\matT_M) \\
& + \sum_{h=0}^\infty \phi(\matT_{\matL}(s) \,\matT_M^h(s) \, \matZ \, \matT_M^{-h}(s)\, \matT_{\matL}^{-1}(s),m+1,s) 
\, \matT_{\matL}(s)  \, \matT_M^h(s) \, \matZ^h \ . \nonumber
\end{align}
The recursive equation \eqref{eq:phi.rec.inf.insp} is valid for $m = 0, 1, \ldots, K-1$.

We conjecture that for all $m = 0, 1, \ldots, K$, the generating function $\phi(\matZ,m,s)$ has the form
\begin{align}\label{eq:phi.conj.insp}
\phi(\matZ,m,s)
= &  \sum_{k=m}^{K-1} \psi(K,k,s)  \,\matT_M(s) \, \matS(\matZ,Y^{k-m}(s),\matT_M(s)) \\
& + \psi(K,K+1,s) \, \matT(s) \, \matS(\matZ,Y^{K-m}(s),\matT(s)) \ , \nonumber
\end{align}
so that \eqref{eq:phi.rec.insp} follows by showing that the right expression for $Y(s)$ is given by
\begin{equation}\label{eq:Y(s)}
  Y(s) = \matT_{\matL}(s) \, \matS(\matZ, I, \matT_M(s)) \ .
\end{equation}
This conjecture will be proved by induction. We have already shown that it holds for $m=K$. 
Now assume that it is valid for $m+1$. To establish \eqref{eq:phi.conj.insp} for $m$, it suffices to prove, by virtue of \eqref{eq:phi.rec.inf.insp}, that \begin{align}\label{eq:ind.0.insp}
&\sum_{h=0}^\infty \phi(\matT_{\matL}(s) \,\matT_M^h(s) \, \matZ \, \matT_M^{-h}(s)\, \matT_{\matL}^{-1}(s),m+1,s)
\matT_{\matL}(s) \, \matT_M^h(s)  \, \matZ^h  \nonumber \\
& = \sum_{k=m+1}^{K-1} \psi(K,k,s) \, \matT_M(s) \, \matS(\matZ,Y^{k-m}(s),\matT_M(s)) \\
& \quad \quad \quad + \psi(K,K+1,s) \, \matT(s) \, \matS(\matZ,Y^{K-m}(s),\matT(s)) \ . \nonumber
\end{align}
It follows from
Lemma \ref{lm:rel.S} that
\begin{align}\label{eq:ind.1.insp}
 \psi(K,K+1) \, \matT  \sum_{h=0}^\infty & 
  \matS(\matT_{\matL} \,\matT_M^h \, \matZ \, \matT_M^{-h}\, \matT_{\matL}^{-1},
    Y^{K-m-1},\matT) \, \matT_{\matL} \, \matT_M^h  \, \matZ^h \nonumber \\
& = \psi(K,K+1) \, \matT \matS(\matZ, \, Y^{K-m-1} \, \matT_{\matL} \, \matS(\matZ, I, \matT_M)  , \matT) \ ,
\end{align}
where we suppressed the dependence on $s$. Application of Lemma \ref{lm:rel.S} is justified, since it 
is readily verified that the matrices in the above infinite sum satisfy the conditions mentioned in this lemma.
Accordingly, for $k=m+1,\ldots,K-1$, and again suppressing the dependence on $s$,
\begin{align}\label{eq:ind.2.insp}
\psi(K,k) \, \matT_M 
\sum_{h=0}^\infty & \matS(\matT_{\matL} \,\matT_M^h \, \matZ \, \matT_M^{-h}\, \matT_{\matL}^{-1},
   Y^{k-m-1},\matT_M) \, \matT_{\matL} \, \matT_M^h \, \matZ^h \nonumber \\
& = \psi(K,k) \, \matT_M \, \matS(\matZ,  Y^{k-m-1}  \, \matT_{\matL} \, \matS(\matZ, I, \matT_M) , \matT_M) \ .
\end{align}
Combining \eqref{eq:ind.1.insp} and \eqref{eq:ind.2.insp} we conclude, by virtue of the induction hypothesis, that \eqref{eq:ind.0.insp}
holds whenever $Y(s)$ satisfies \eqref{eq:Y(s)}, which completes the proof. \end{proof}

\begin{remark}\label{rm:exp.sw.closed-formulas}Also in this case, as was already mentioned in Remark \ref{rm:inv.lap.scalar}, the Laplace transform of the sojourn time is rational.
This admits application of classical inversion techniques, yielding an explicit expression for the sojourn time distribution.
In Section \ref{anal-example} we give an example of how to compute the density function of the sojourn time for a system with  $K=2$.
\end{remark}

%% file: pic-model-2.tex
\begin{figure*}[h]\centering
\newif\ifrepONE
\repONEtrue 
\newcommand{\nodet}[2]{%
  \ifrepONE #1 \else
   \dfrac{#1}%
    {\hbox{\hspace*{0.25cm}#2\hspace*{0.25cm}}}
\fi}
\newcommand{\labelt}[1]{\ifrepONE #1 \fi}
\begin{tikzpicture}[->,shorten >=1pt,node distance=.75cm,
state/.style={circle,draw,minimum size=0.8cm}]
\tiny 
\ifrepONE
  \tikzstyle{label}=[minimum size=0.8cm] 
  \node[label] (L) [] {\labelt{Low rate:}};
  \node[label] (H) [below=of L] {\labelt{High rate:}};
  \node[state] (0L) [right=of L] {$\nodet{0}{L}$};
  \node[state] (0H) [right=of H] {$\nodet{0}{H}$};
\else
  \node[state] (0L) [] {$\nodet{0}{L}$};
  \node[state] (0H) [below=of 0L] {$\nodet{0}{H}$};
\fi
\node[state] (1L) [right=of 0L] {$\nodet{1}{L}$};
\node[state,draw=none] (2L) [right=of 1L] {$\ldots$};
\node[state] (KL) [right=of 2L] {$\nodet{K}{L}$};
\node[state] (K+1L) [right=of KL] {$\nodet{K+1}{L}$};
\node[state] (K+2L) [right=of K+1L] {$\nodet{K+2}{L}$};
\node[state,draw=none] (K+3L) [right=of K+2L] {$\ldots$};
\node[state] (1H) [right=of 0H] {$\nodet{1}{H}$};
\node[state,draw=none] (2H) [right=of 1H] {$\ldots$};
\node[state] (KH) [right=of 2H] {$\nodet{K}{H}$};
\node[state] (K+1H) [right=of KH] {$\nodet{K+1}{H}$};
\node[state] (K+2H) [right=of K+1H] {$\nodet{K+2}{H}$};
\node[state,draw=none] (K+3H) [right=of K+2H] {$\ldots$};
\path [every node/.style={font=\tiny}]
	(0L)  edge [bend left] node [above] {$\lambda$} (1L);
\path [every node/.style={font=\tiny}]
	(1L)	 edge [bend left] node [above] {$\lambda$} (2L)
	     edge [bend left] node [below] {$\mu_0$} (0L)
	(2L)	 edge [bend left] node [above] {$\lambda$} (KL)
	     edge [bend left] node [below] {$\mu_0$} (1L)
	(KL)    edge [bend left] node [above] {$\lambda$}(K+1L)
	       edge [bend left] node [below] {$\mu_0$} (2L)
	(K+1L)  edge [bend left] node [above] {$\lambda$}(K+2L)
	     edge [bend left] node [below] {$\mu_0$} (KL)
	(K+2L)  edge [bend left] node [above] {$\lambda$}(K+3L)
	       edge [bend left] node [below] {$\mu_0$} (K+1L)
	(K+3L)  edge [bend left] node [below] {$\mu_0$} (K+2L);
\path [every node/.style={font=\tiny}]
	(0H)  edge [bend left] node [above] {$\lambda$} (1H);
\path [every node/.style={font=\tiny}]
	(1H)	 edge [bend left] node [above] {$\lambda$} (2H)
	     edge [bend left] node [below] {$\mu_1$} (0H)
	(2H)	 edge [bend left] node [above] {$\lambda$} (KH)
	     edge [bend left] node [below] {$\mu_1$} (1H)
	(KH)    edge [bend left] node [above] {$\lambda$}(K+1H)
	       edge [bend left] node [below] {$\mu_1$} (2H)
	(K+1H)  edge [bend left] node [above] {$\lambda$}(K+2H)
	     edge [bend left] node [below] {$\mu_1$} (KH)
	(K+2H)  edge [bend left] node [above] {$\lambda$}(K+3H)
	       edge [bend left] node [below] {$\mu_1$} (K+1H)
	(K+3H)  edge [bend left] node [below] {$\mu_1$} (K+2H);
\path [every node/.style={font=\tiny}]
	(K+1L)  edge [bend left] node [right] {$\gamma$} (K+1H)
	(K+2L)  edge [bend left] node [right] {$\gamma$} (K+2H);
\path [every node/.style={font=\tiny}]
	(0H)  edge [bend left] node [right] {$\gamma$} (0L)
	(1H)  edge [bend left] node [right] {$\gamma$} (1L)
	(KH)  edge [bend left] node [right] {$\gamma$} (KL);
\end{tikzpicture}
\caption{Transition diagram for exponential inspection times}
\label{fig:trans.diag.inspections}       
\end{figure*}
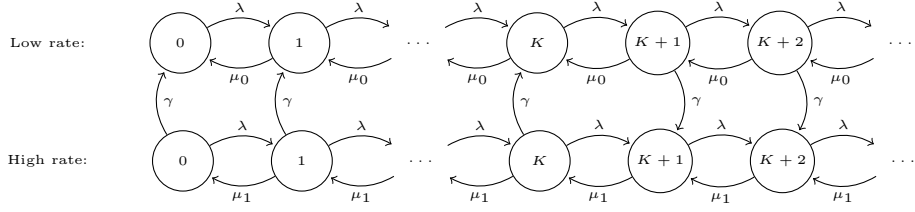

%% file: Erlang-example.tex
\subsection{Erlang inspection times}
\label{Erlang-example}
In section \ref{sc:inspection-times} we assumed exponential inter-inspection times. 
In principle this can be extended to the case of phase-type distributed inter-inspection times \cite{asmussen}, paying a cost in terms of model complexity. 
Indeed, in this case one should keep track, not only of the value of the service rate, but also of the phase of the inspection-clock. This translates into more complicated matrix expressions, but the basic logic of the computation of the sojourn time distribution remains the same. 
In fact, this is exactly the power of the proposed matrix generating function technique. 
For the sake of clarity and conciseness we are not going to treat here this extension in detail, but give a quick view of how it can be handled.

We assume that the inspection times are Erlang(2,$\gamma$) distributed. To keep trace of this we consider four states in the description of the system, $\{00,01,10,11\}$, 
where the first number specifies the speed of the system  and the second the phase of the inspection clock. 

\input{pic-model-3.tex}

The column vector $\pi_n=(\pi_{00n},\ldots,\pi_{11n})^\tr$ satisfies \eqref{eq:insp.stat.dist.rel} with the following matrices
\def\MATI{\left(\begin{array}{cc} 1 & 0 \\ 0 & 1 \end{array}\right)}
\begin{equation*}
\matM =  \left(\begin{array}{cc} \mu_0 & 0 \\ 0 & \mu_1 \end{array}\right) \otimes \MATI \ , \quad
\matL =   \left(\begin{array}{cc} \lambda & 0 \\ 0 & \lambda \end{array}\right) \otimes \MATI  \ ,
\end{equation*}
and $H_1 = \matL + \matG_2$,
$H_2 = \matM + \matL + \matG_2$
and $H_3 = \matM + \matL + \matG_3$, where
\begin{equation*}
\matG_2 =  \left(\begin{array}{cccc} 
\gamma & -\gamma & 0 & -\gamma \\
-\gamma & \gamma & 0 & 0  \\
0 & 0 & \gamma & 0 \\
0 & 0 & -\gamma & \gamma 
\end{array}\right) \ , \quad
\matG_3 =  \left(\begin{array}{cccc} 
\gamma & 0 & 0 & 0 \\
-\gamma & \gamma & 0 & 0  \\
0 & -\gamma & \gamma & -\gamma \\
0 & 0 & -\gamma & \gamma 
\end{array}\right) \ . \quad
\end{equation*}
The conditional sojourn times satisfy the following equation, see the corresponding formula \eqref{eq:soj.rel.insp},
\begin{align}\label{eq:soj.rel.insp.erl}
S_{ij}(n,m) &=
\frac{X}{\lambda+\mu_i+\gamma} +
 \left\{\begin{array}{ll}
S_{i,j}(n-1,m) & \mbox{w.p. } \mu_i/(\lambda+\mu_i+\gamma)  \\
S_{i,j}(n,m+1) & \mbox{w.p. } \lambda/(\lambda+\mu_i+\gamma)  \\
S_{h(i,j)}(n,m)  & \mbox{w.p. } \gamma/(\lambda+\mu_i+\gamma)
\end{array}
\right.
\end{align}
with $h(i,j)=h(i,j;n,m)=((1-j) \cdot i+j\cdot 1\{n + m > K\},(1-j))$.
It follows that the row vector $(\psi_{00}(n,m,s),\psi_{01}(n,m,s),\psi_{10}(n,m,s),\psi_{11}(n,m,s))$ satisfies
equation \eqref{eq:lap.rel.insp} with the matrices
$H(s) = (s+\gamma) \, \matI + \matL +\matM$
and 
\begin{equation*}
\matG_0 =  \left(\begin{array}{cccc} 
0 & \gamma & 0 & \gamma \\
\gamma & 0 & 0 & 0  \\
0 & 0 & 0 & 0 \\
0 & 0 & \gamma & 0 
\end{array}\right) \ , \quad
\matG_1 =  \left(\begin{array}{cccc} 
0 & 0 & 0 & 0 \\
\gamma & 0 & 0 & 0  \\
0 & \gamma & 0 & \gamma \\
0 & 0 & \gamma & 0 
\end{array}\right)  \ . \quad
\end{equation*}

Since the matrix equations for the Erlang inspection times are similar to the exponential inspection times,
all the subsequent matrix analysis in Proposition \ref{prop:lap.trans} and Theorem \ref{th:phi.rec.insp} are still valid.
The value of the matrix $R$ is now given by
$$R=
\left(
\begin{array}{cccc}
 R_{11} & 0 & 0 & 0 \\
 R_{21} & R_{11} & 0 & 0 \\
 R_{31} & R_{32} & R_{33} & R_{34} \\
 R_{41} & R_{42} & R_{34} & R_{33}
\end{array}
\right)
$$
with 
\begin{eqnarray*}
R_{11} &=& \frac{\gamma +\lambda +\mu_0-\sqrt{-4 \lambda  \mu_0+(\gamma +\lambda +\mu_0)^2}}{2 \mu_0} \ , \quad 
R_{21} = \frac{\gamma R_{11}}{\gamma +\lambda +\mu_0-2 \mu_0R_{11}} \ ; \\
R_{33} &=& \frac{\mu_1 (2 \gamma +3 \lambda +\mu_1)-\sqrt{\mu_1^2 \left(4 \gamma ^2+(\lambda -\mu_1)^2
                        +4 \gamma  (\lambda +\mu_1)\right)}}{4 \mu_1^2} \ ; \\
R_{34} &=& \frac{\lambda}{\mu_1 } - R_{33} \ , \quad
R_{31} = \frac{\gamma  (R_{41}+R_{21})+\mu_1 (R_{32} R_{21}+R_{41} R_{34})}{\gamma +\lambda -\mu_1 (-1+R_{11}+R_{33})}  \ ; \\
R_{32} &=& \frac{\gamma  R_{11} (-\gamma -\lambda +\mu_1 (-1+R_{11}+R_{33}))}{-(\gamma +\lambda -\mu_1 (-1+R_{11}+R_{33}))^2
+(\gamma + \mu_1 R_{34})^2}  \ ; \\
R_{41} &=& \frac{\gamma  R_{21} (\gamma +\mu_1 R_{34}) \left(\lambda ^2-2 \lambda  \mu_1 (-1+R_{33})-\mu_1^2 \left(R_{11}^2-(-1+R_{33})^2+R_{34}^2\right)\right)}{(2 \gamma +\lambda -\mu_1 (-1+R_{11}+R_{33}-R_{34}))^2 (\lambda -\mu_1 (-1+R_{11}+R_{33}+R_{34}))^2} \\
& & \frac{\gamma  R_{21} (\gamma +\mu_1 R_{34}) \left(2 \gamma  (\lambda -\mu_1 (-1+R_{33}+R_{34}))\right)}{(2 \gamma +\lambda -\mu_1 (-1+R_{11}+R_{33}-R_{34}))^2 (\lambda -\mu_1 (-1+R_{11}+R_{33}+R_{34}))^2} \ ; \\
R_{42} &=& \frac{\gamma  R_{11} (\gamma +\mu_1 R_{34})}{(2 \gamma +\lambda -\mu_1 (-1+R_{11}+R_{33}-R_{34})) (\lambda -\mu_1 (-1+R_{11}+R_{33}+R_{34}))} \ .
\end{eqnarray*}

%% file: pic-model-3.tex
\begin{figure*}[h]\centering
\newif\ifrepONE
\repONEtrue 
\newcommand{\nodet}[2]{%
  \ifrepONE #1 \else
   \dfrac{#1}%
    {\hbox{\hspace*{0.25cm}#2\hspace*{0.25cm}}}
\fi}
\newcommand{\labelt}[1]{\ifrepONE #1 \fi}
\begin{tikzpicture}[->,shorten >=1pt,node distance=.75cm,
state/.style={circle,draw,minimum size=0.8cm}]
\tiny 
\ifrepONE
  \tikzstyle{label}=[minimum size=0.8cm] 
  \node[label] (La) [] {\labelt{Low rate (a):}};
  \node[label] (Lb) [below=of La] {\labelt{Low rate (b):}};
  \node[label] (Ha) [below=of Lb] {\labelt{High rate (a):}};
  \node[label] (Hb) [below=of Ha] {\labelt{High rate (b):}};
  \node[state] (0La) [right=of La] {$\nodet{0}{La}$};
  \node[state] (0Lb) [right=of Lb] {$\nodet{0}{Lb}$};
  \node[state] (0Ha) [right=of Ha] {$\nodet{0}{Ha}$};
  \node[state] (0Hb) [right=of Hb] {$\nodet{0}{Hb}$};
\else
  \node[state] (0La) [] {$\nodet{0}{La}$};
  \node[state] (0Lb)  [below=of 0La] {$\nodet{0}{Lb}$};
  \node[state] (0Ha) [below=of 0Lb] {$\nodet{0}{Ha}$};
  \node[state] (0Hb) [below=of 0Ha] {$\nodet{0}{Hb}$};
\fi
\node[state] (1La) [right=of 0La] {$\nodet{1}{La}$};
\node[state,draw=none] (2La) [right=of 1La] {$\ldots$};
\node[state] (KLa) [right=of 2La] {$\nodet{K}{La}$};
\node[state] (K+1La) [right=of KLa] {$\nodet{K+1}{La}$};
\node[state] (K+2La) [right=of K+1La] {$\nodet{K+2}{La}$};
\node[state,draw=none] (K+3La) [right=of K+2La] {$\ldots$};
\node[state] (1Lb) [right=of 0Lb] {$\nodet{1}{Lb}$};
\node[state,draw=none] (2Lb) [right=of 1Lb] {$\ldots$};
\node[state] (KLb) [right=of 2Lb] {$\nodet{K}{Lb}$};
\node[state] (K+1Lb) [right=of KLb] {$\nodet{K+1}{Lb}$};
\node[state] (K+2Lb) [right=of K+1Lb] {$\nodet{K+2}{Lb}$};
\node[state,draw=none] (K+3Lb) [right=of K+2Lb] {$\ldots$};
\node[state] (1Ha) [right=of 0Ha] {$\nodet{1}{Ha}$};
\node[state,draw=none] (2Ha) [right=of 1Ha] {$\ldots$};
\node[state] (KHa) [right=of 2Ha] {$\nodet{K}{Ha}$};
\node[state] (K+1Ha) [right=of KHa] {$\nodet{K+1}{Ha}$};
\node[state] (K+2Ha) [right=of K+1Ha] {$\nodet{K+2}{Ha}$};
\node[state,draw=none] (K+3Ha) [right=of K+2Ha] {$\ldots$};
\node[state] (1Hb) [right=of 0Hb] {$\nodet{1}{Hb}$};
\node[state,draw=none] (2Hb) [right=of 1Hb] {$\ldots$};
\node[state] (KHb) [right=of 2Hb] {$\nodet{K}{Hb}$};
\node[state] (K+1Hb) [right=of KHb] {$\nodet{K+1}{Hb}$};
\node[state] (K+2Hb) [right=of K+1Hb] {$\nodet{K+2}{Hb}$};
\node[state,draw=none] (K+3Hb) [right=of K+2Hb] {$\ldots$};
\path [every node/.style={font=\tiny}]
	(0La)  edge [bend left] node [above] {$\lambda$} (1La);
\path [every node/.style={font=\tiny}]
	(1La)	 edge [bend left] node [above] {$\lambda$} (2La)
	     edge [bend left] node [below] {$\mu_0$} (0La)
	(2La)	 edge [bend left] node [above] {$\lambda$} (KLa)
	     edge [bend left] node [below] {$\mu_0$} (1La)
	(KLa)    edge [bend left] node [above] {$\lambda$}(K+1La)
	       edge [bend left] node [below] {$\mu_0$} (2La)
	(K+1La)  edge [bend left] node [above] {$\lambda$}(K+2La)
	     edge [bend left] node [below] {$\mu_0$} (KLa)
	(K+2La)  edge [bend left] node [above] {$\lambda$}(K+3La)
	       edge [bend left] node [below] {$\mu_0$} (K+1La)
	(K+3La)  edge [bend left] node [below] {$\mu_0$} (K+2La);
\path [every node/.style={font=\tiny}]
	(0Lb)  edge [bend left] node [above] {$\lambda$} (1Lb);
\path [every node/.style={font=\tiny}]
	(1Lb)	 edge [bend left] node [above] {$\lambda$} (2Lb)
	     edge [bend left] node [below] {$\mu_0$} (0Lb)
	(2Lb)	 edge [bend left] node [above] {$\lambda$} (KLb)
	     edge [bend left] node [below] {$\mu_0$} (1Lb)
	(KLb)    edge [bend left] node [above] {$\lambda$}(K+1Lb)
	       edge [bend left] node [below] {$\mu_0$} (2Lb)
	(K+1Lb)  edge [bend left] node [above] {$\lambda$}(K+2Lb)
	     edge [bend left] node [below] {$\mu_0$} (KLb)
	(K+2Lb)  edge [bend left] node [above] {$\lambda$}(K+3Lb)
	       edge [bend left] node [below] {$\mu_0$} (K+1Lb)
	(K+3Lb)  edge [bend left] node [below] {$\mu_0$} (K+2Lb);
\path [every node/.style={font=\tiny}]
	(0Ha)  edge [bend left] node [above] {$\lambda$} (1Ha);
\path [every node/.style={font=\tiny}]
	(1Ha)	 edge [bend left] node [above] {$\lambda$} (2Ha)
	     edge [bend left] node [below] {$\mu_1$} (0Ha)
	(2Ha)	 edge [bend left] node [above] {$\lambda$} (KHa)
	     edge [bend left] node [below] {$\mu_1$} (1Ha)
	(KHa)    edge [bend left] node [above] {$\lambda$}(K+1Ha)
	       edge [bend left] node [below] {$\mu_1$} (2Ha)
	(K+1Ha)  edge [bend left] node [above] {$\lambda$}(K+2Ha)
	     edge [bend left] node [below] {$\mu_1$} (KHa)
	(K+2Ha)  edge [bend left] node [above] {$\lambda$}(K+3Ha)
	       edge [bend left] node [below] {$\mu_1$} (K+1Ha)
	(K+3Ha)  edge [bend left] node [below] {$\mu_1$} (K+2Ha);
\path [every node/.style={font=\tiny}]
	(0Hb)  edge [bend left] node [above] {$\lambda$} (1Hb);
\path [every node/.style={font=\tiny}]
	(1Hb)	 edge [bend left] node [above] {$\lambda$} (2Hb)
	     edge [bend left] node [below] {$\mu_1$} (0Hb)
	(2Hb)	 edge [bend left] node [above] {$\lambda$} (KHb)
	     edge [bend left] node [below] {$\mu_1$} (1Hb)
	(KHb)    edge [bend left] node [above] {$\lambda$}(K+1Hb)
	       edge [bend left] node [below] {$\mu_1$} (2Hb)
	(K+1Hb)  edge [bend left] node [above] {$\lambda$}(K+2Hb)
	     edge [bend left] node [below] {$\mu_1$} (KHb)
	(K+2Hb)  edge [bend left] node [above] {$\lambda$}(K+3Hb)
	       edge [bend left] node [below] {$\mu_1$} (K+1Hb)
	(K+3Hb)  edge [bend left] node [below] {$\mu_1$} (K+2Hb);
\path [every node/.style={font=\tiny}]
	(0La)  edge [bend left] node [right] {$\gamma$} (0Lb)
	(1La)  edge [bend left] node [right] {$\gamma$} (1Lb)
	(KLa)  edge [bend left] node [right] {$\gamma$} (KLb)
	(K+1La)  edge [bend left] node [right] {$\gamma$} (K+1Lb)
	(K+2La)  edge [bend left] node [right] {$\gamma$} (K+2Lb);
\path [every node/.style={font=\tiny}]
	(0Lb)  edge [bend left] node [right] {$\gamma$} (0La)
	(1Lb)  edge [bend left] node [right] {$\gamma$} (1La)
	(KLb)  edge [bend left] node [right] {$\gamma$} (KLa)
	(K+1Lb)  edge [bend left] node [right] {$\gamma$} (K+1Ha)
	(K+2Lb)  edge [bend left] node [right] {$\gamma$} (K+2Ha);
\path [every node/.style={font=\tiny}]
	(0Ha)  edge [bend left] node [right] {$\gamma$} (0Hb)
	(1Ha)  edge [bend left] node [right] {$\gamma$} (1Hb)
	(KHa)  edge [bend left] node [right] {$\gamma$} (KHb)
	(K+1Ha)  edge [bend left] node [right] {$\gamma$} (K+1Hb)
	(K+2Ha)  edge [bend left] node [right] {$\gamma$} (K+2Hb);
\path [every node/.style={font=\tiny}]
	(0Hb)  edge [bend left] node [right] {$\gamma$} (0La)
	(1Hb)  edge [bend left] node [right] {$\gamma$} (1La)
	(KHb)  edge [bend left] node [right] {$\gamma$} (KLa)
	(K+1Hb)  edge [bend left] node [right] {$\gamma$} (K+1Ha)
	(K+2Hb)  edge [bend left] node [right] {$\gamma$} (K+2Ha);
\end{tikzpicture}
\caption{Transition diagram for Erlang-2 inspection times}
\label{fig:trans.diag.Erlang.inspections}       
\end{figure*}
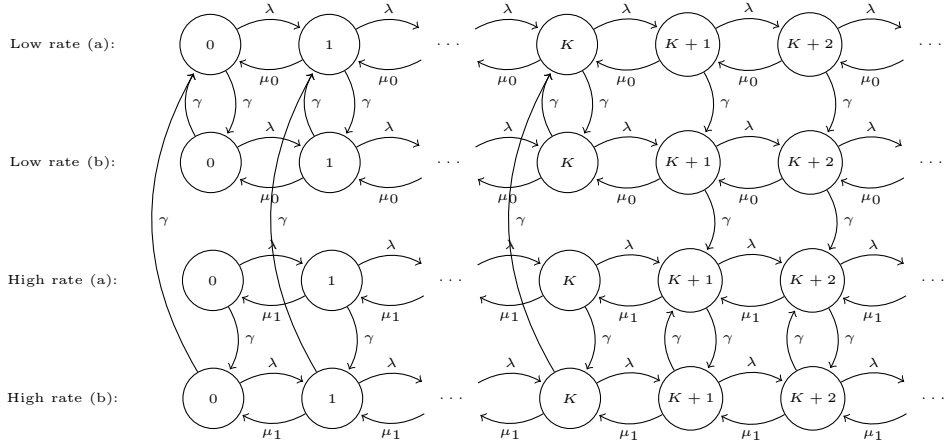

%% file: anal-example.tex
\subsection{Analitical example}
\label{anal-example}
\definecolor{curveONE}{rgb}{0.2472, 0.24, 0.6}
\definecolor{curveTWO}{rgb}{0.6, 0.24, 0.442893}
\definecolor{curveTHREE}{rgb}{0.6, 0.547014, 0.24}
\definecolor{curveFOUR}{rgb}{0.24, 0.6, 0.33692}
\definecolor{curveFIVE}{rgb}{0.24, 0.353173, 0.6}
\definecolor{curveSIX}{rgb}{0.6, 0.24, 0.563266}
\definecolor{curveSEVEN}{rgb}{0.6, 0.426641, 0.24}
\definecolor{curveEIGHT}{rgb}{0.263452, 0.6, 0.24}
\definecolor{curveNINE}{rgb}{0.24, 0.473545, 0.6}
\definecolor{curveTEN}{rgb}{0.516361, 0.24, 0.6}

In this section we briefly show that by using Theorem \ref{th:phi.rec.insp}, we can get explicit expressions for the density function of the sojourn time in the system with inspection times, as highlighted in Remark \ref{rm:exp.sw.closed-formulas}.

The computations are simple, but tedious as they require extensive use of matrix calculus, and usually it is easy to be assisted by symbolic computational software as we do for this example. 

To make computations easy, we wisely select the values of the parameters of the system such that all coefficients turn out to be rational.

The parameters of the queue are
$$\mu_0 = 1; \quad  \mu_1 = 3/2; \quad \gamma = 1/8; \quad  \lambda=  9/8 \ .$$ %
For the moment we do not fix the threshold, later we consider explicitly the case $K=2$.
The above choice of the parameters gives  $R_{00} = 3/4$ in \eqref{eq:matrix.R}. The matrix $R$ and the matrix function  
$\matT(s)= \matM \, (H(s)-\matGr-\matL)^{-1}$, are given by 
\begin{equation*}
R = \left(
\begin{array}{cc}
 \frac{3}{4} & 0 \\
 \frac{1}{4} & \frac{3}{4}
\end{array}
\right)
; \quad 
\matT(s) = \left(
\begin{array}{cc}
 \frac{8}{9+8 s} & 0 \\
 \frac{3}{(3+2 s) (9+8 s)} & \frac{3}{3+2 s}
\end{array}
\right)
\end{equation*}
and the matrix functions 
$\matT_\matL(s) = \matL \, (H(s)-\matGr)^{-1}$ and 
$\matT_\matM(s) = \matM \, (H(s)-\matGr)^{-1}$
are equal to
$$
\matT_\matL(s) =\left(
\begin{array}{cc}
 \frac{9}{2 (9+4 s)} & 0 \\
 \frac{9}{2 (9+4 s) (21+8 s)} & \frac{9}{21+8 s}
\end{array}
\right)
; \quad 
\matT_\matM(s) = \left(
\begin{array}{cc}
 \frac{4}{9+4 s} & 0 \\
 \frac{6}{(9+4 s) (21+8 s)} & \frac{12}{21+8 s}
\end{array}
\right) \ .
$$

Solving the matrix system \eqref{eq:S.syst} we get the following expression for 
$\matS(R,I,\matT_M(s))$,
$$
\matS(R,I,\matT_M(s))
=\left(
\begin{array}{cc}
 \frac{9+4 s}{2 (3+2 s)} & 0 \\
 \frac{3 (3+s)}{2 (3+2 s)^2} & \frac{21+8 s}{4 (3+2 s)}
\end{array}
\right)$$
that allows the computation of the values of $\matU(R,k,s)$ and $\matU_M(R,k,s)$ for any $k\geq0$. As example we show such matrix functions for $k=2$, 
$$
\matU(R,2,s) =
\left(
\begin{array}{cc}
 \frac{81 (9+8 s)}{16 (3+2 s)^2 (3+8 s)} & 0 \\
 \frac{81 (69+88 s)}{16 (3+2 s)^2 (3+8 s)^2} & \frac{81}{4 (3+2 s) (3+8 s)}
\end{array}
\right)
; \quad 
\matU_M(R,2,s) =
\left(
\begin{array}{cc}
 \frac{81 (9+4 s)}{32 (3+2 s)^3} & 0 \\
 \frac{81 (30+11 s)}{32 (3+2 s)^4} & \frac{81 (21+8 s)}{64 (3+2 s)^3}
\end{array}
\right) \ .
$$
\begin{remark} The expressions of $\matS(R,I,\matT_M(s)))$, $\matU(R,k,s)$ and $\matU_M(R,k,s)$ do not depend on $K$, so they can be used for any value of the threshold. 
The values of $\psi(s)$, $\psi(n,0,s)$, $\pi_n$ and $\phi(R,0,s)$ in \eqref{eq:Laplace.eq.insp} do depend on $K$ via the respective formulas  \eqref{eq:Laplace.eq.insp}, \eqref{eq:lap.rel.insp}, \eqref{eq:norm.eq} and \eqref{eq:phi.rec.insp}.
\end{remark}

From here on we fix $K=2$. We have $\pi_K = (3807/60644, 1701/30322)^\tr$ and after recursively computing $\psi(k,0,s)$, for $k=1,2$, we finally get  $\psi(s)$,
\begin{align*}
\psi(s) =& -\frac{308367}{379025 (3+2 s)^4}-\frac{13923657}{9475625 (3+2 s)^3}-\frac{44764461}{47378125 (3+2 s)^2} -\frac{130808703}{236890625 (3+2 s)}\\
&-\frac{14013}{15161 (9+4 s)}+\frac{1587762}{9475625 (11+4 s)^3}-\frac{4755267}{24636625 (11+4 s)^2} -\frac{28797784929}{40034515625 (11+4 s)} \\
&+\frac{81216}{15161 (3+8 s)^2}+\frac{18144}{15161 (3+8 s)}+\frac{102060}{15161 (9+8 s)^2} +\frac{55081053}{20497672 (9+8 s)} \\
&-\frac{24064452}{9475625 (17+8 s)^3}-\frac{199526994}{47378125 (17+8 s)^2}+\frac{2950774277}{1895125000 (17+8 s)}+\frac{90111}{60644 (21+8 s)}
\end{align*}
the inverse-transform of which results into the following density function
\begin{align*}
f(t) &= -\frac{90111 e^{-21 t/8}}{485152}-\frac{14013 e^{-9 t/4}}{60644}+\frac{27 e^{-3 t/8} (84+47 t)}{15161}+\frac{729 e^{-9 t/8} (75557+23660 t)}{163981376}\\
&+{243 e^{-11 t/4} \left(-1896150448-127198500 t+13803075 t^2\right)}/{2562209000000}\\
&-{e^{-17 t/8} \left(-11803097108+3990539880 t+150402825 t^2\right)}/{60644000000}\\
&-{3 e^{-3 t/2} \left(697646416+596859480 t+232060950 t^2+21414375 t^3\right)}/{7580500000} \ .
\end{align*}
Figure \ref{fig:closed-form} plots the density functions of the sojourn time for $K=0,1,2,3$ using their exact expressions, 
instead of using numeric inverse transform as done later on in Section \ref{num-experiments}. 
\begin{figure}[h]%
\centering
\includegraphics[width=0.65\textwidth]{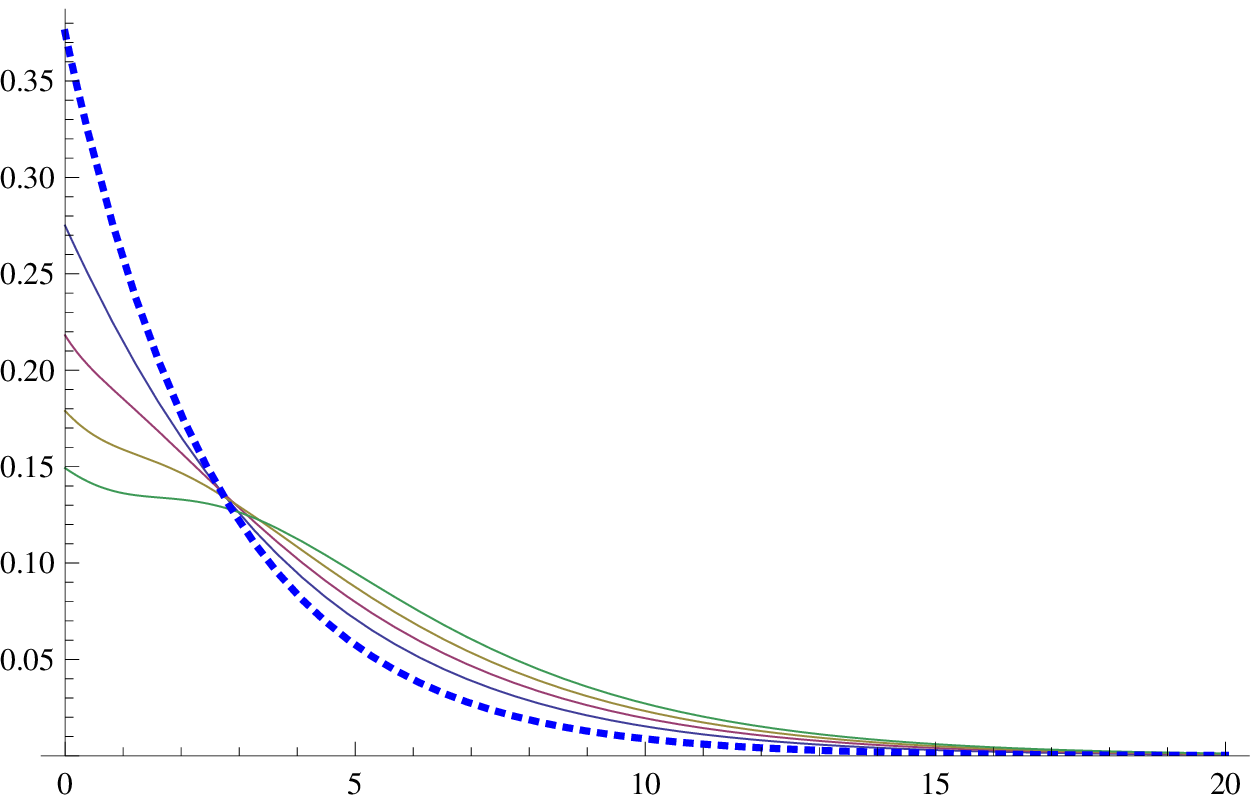}
  \caption{$\mu_0 = 1; \quad  \mu_1 = 3/2; \quad \gamma = 1/8; \quad  \lambda=  9/8;$ and  
$K\in\{%
{\color{curveONE} 0},
{\color{curveTWO} 1},
{\color{curveTHREE} 2},
{\color{curveFOUR} 3}
 \}$}
  \label{fig:closed-form}
\end{figure}%

%% file: model-inspections-mom.tex
\subsection{First moment calculation}
Like in the previous section, we define  $\nu = \E[S]$ and $\nu_{n,m}=\E[S(n,m)]$.
By taking derivatives in \eqref{eq:lap.rel.insp} and then computing the limit for $s\to0$ we get
\begin{align}\label{eq:av.lap.rel.insp}
\nu_{n,m} \, (\matH(0)-\matG_{1\{n+m>K\}})
= \nu_{n-1,m}  \, \matM  + \nu_{n,m+1} \, \matL + e \ ,
\end{align}
where we used that $H'(0)$ is the identity matrix.
The vector $e$ is the all-one vector.
From \eqref{eq:only.H.insp} and after taking derivatives, we obtain
\begin{equation}\label{eq:av.only.H.insp}
\nu_{n,m} = e \, \sum_{k=1}^n (\matT(0))^k \, \matM^{-1} \, (\matT(0))^{n-k+1}
  \quad \mbox{as } n > 0 \mbox{ and } m \geq K \ ,
\end{equation}
with
$\matT(0) = \matM \, (\matH(0) -\Gamma_1-\matL)^{-1}$,
$(\matT^{-1})'(0) = \matM^{-1}$
and
$\matT'(0) = -\matT(0) \matM^{-1} \matT(0)$. Here we used that the derivative of a matrix $\matA^{-n}$ is given by
\[
(\matA^{-n})' = \sum_{k=1}^n \matA^{-k} \, \matA' \matA^{k-n-1} \ .
\] 
By Proposition \ref{prop:lap.trans} we can conclude that 
\begin{equation}\label{eq:av.Laplace.eq.insp}
\nu =  \sum_{n=0}^{K-1} \nu_{n+1,0}  \, \pi_n
  -\phi'(\matZ,0,0+)  \, \pi_K \ .
\end{equation} 
From equation \eqref{eq:phi.rec.insp} we can compute
\begin{align}
-\phi'(\matZ,0,0+) = & \sum_{k=0}^{K-1} \nu_{K,k} \,\matT_M(0) \,\matU_M(\matZ,k,0) 
+ \nu_{K,K+1}  \,  \matT(0) \, \matU(\matZ,K,0) \nonumber \\
& - e \, \sum_{k=0}^{K-1} \matT_M(0) \, \matU_M'(\matZ,k,0)  - e \, \matT(0) \, \matU'(\matZ,K,0) \nonumber \\
& - e \, \sum_{k=0}^{K-1} \matT'_M(0) \, \matU_M(\matZ,k,0) - e \, \matT'(0) \, \matU(\matZ,K,0) 
 \ ,
\end{align}
with $\matT_M(0) = \matM \, (\matH(0)-\Gamma_1)^{-1} \, $
and $\matT'_M(0) = \matT_M(0)  \, \matM^{-1} \ \matT_M(0)$.
The values $\matU'_M(\matZ,k,0)$ and $\matU'(\matZ,k,0)$ 
appearing in \eqref{eq:phi.rec.insp} can be computed by solving the following linear systems, see Lemma \ref{lm:S.der} in the appendix,
\begin{align*}
\matU_M'(\matZ,k,0)   - \matT_M(0)  \, \matU_M'(\matZ,k,0)   \,  \matZ  -  \matT_M'(0) \, \matU_M(\matZ,k,0) \, \matZ &
  = \matA'(\matZ,k,0)  \\
\matU'(\matZ,k,0)  -  \matT(0)  \, \matU'(\matZ,k,0) \, \matZ  -  \matT'(0) \, \matU(\matZ,k,0) \, \matZ & = \matA'(\matZ,k,0)
\end{align*}
with
$\matA(\matZ,k,s) =  (\matT_{\matL}(s) \, \matS(\matZ, I, \matT_M(s) ) )^k$.

%% file: num-experiments.tex
\section{Numerical experiments}
\label{num-experiments}
\definecolor{curveONE}{rgb}{0.2472, 0.24, 0.6}
\definecolor{curveTWO}{rgb}{0.6, 0.24, 0.442893}
\definecolor{curveTHREE}{rgb}{0.6, 0.547014, 0.24}
\definecolor{curveFOUR}{rgb}{0.24, 0.6, 0.33692}
\definecolor{curveFIVE}{rgb}{0.24, 0.353173, 0.6}
\definecolor{curveSIX}{rgb}{0.6, 0.24, 0.563266}
\definecolor{curveSEVEN}{rgb}{0.6, 0.426641, 0.24}
\definecolor{curveEIGHT}{rgb}{0.263452, 0.6, 0.24}
\definecolor{curveNINE}{rgb}{0.24, 0.473545, 0.6}
\definecolor{curveTEN}{rgb}{0.516361, 0.24, 0.6}

In this section we show some numerical examples, where we compute the stationary sojourn time distribution
for a system with slow rate $\mu_0=1$ and high rate $\mu_1=3/2$.

In Figures \ref{fig:imm-sw-stable} and \ref{fig:imm-sw-unstable}, it is shown how the sojourn time distribution depends on the threshold $K$ for the case of immediate switching times.
In the first example, $\lambda < \mu_0 < \mu_1$, which implies that the system is stable for both service rates. Therefore, when $K\to\infty$, one can appreciate that the sojourn time distribution approaches the one of an $M/M/1$ system with fixed service rate $\mu_0$ (shown as dashed black line in Figure \ref{fig:imm-sw-stable}).
In the second example, we have $\lambda \in [\mu_0,\mu_1)$. In particular, we have chosen $\lambda=\mu_0=1$, implying that the system approaches instability as $K\to\infty$.

\begin{figure}[h]%
\centering
\begin{minipage}{0.45\textwidth}
\includegraphics[width=\textwidth]{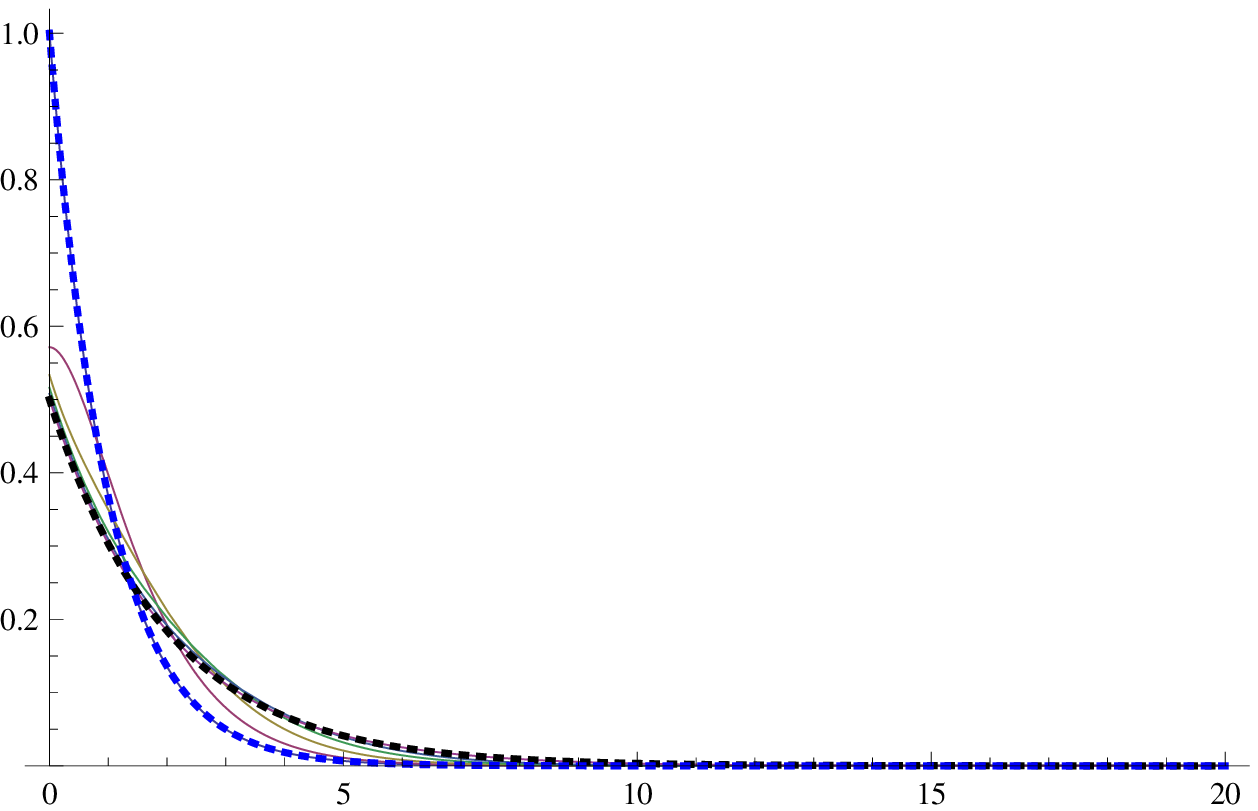}
  \caption{$\lambda = 1/2$}
  \label{fig:imm-sw-stable}
\end{minipage}%
\qquad
\begin{minipage}{0.45\textwidth}
  \includegraphics[width=\textwidth]{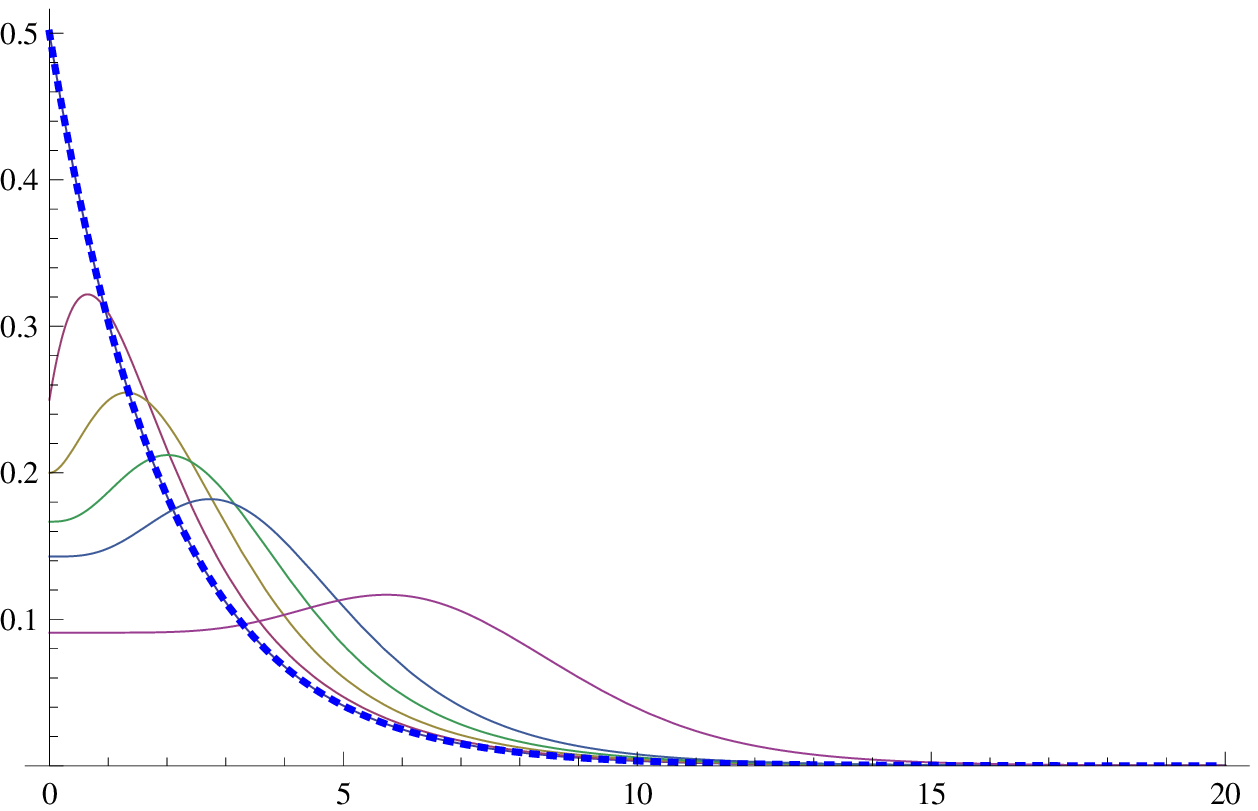}
  \caption{$\lambda = 1$}
  \label{fig:imm-sw-unstable}
\end{minipage}%

\vspace{.25cm}
Sojourn time distribution for: $\mu_0 = 1$, $\mu_1 = 3/2$ and %
 $K\in\{%
{\color{blue} 0},
{\color{curveTWO} 1},
{\color{curveTHREE} 2},
{\color{curveFOUR} 3},
{\color{curveFIVE} 4},
{\color{curveSIX} 8},
{\color{black} \infty}
 \}$
 \label{fig:immediate-switching}
\end{figure}%

Figures \ref{fig:exp-sw-1/100} -- \ref{fig:exp-sw-1000}, depict the sojourn time distribution for the case of exponential distributed inspection times.
These figures refer to the case when  $\lambda \in [\mu_0,\mu_1)$, and again, one can notice that as $K\to\infty$, the system becomes unstable.
It is worth to notice that, when $K=0$, the curve does not coincide with the $M/M/1$ with constant service rate $\mu_1$ (shown as dashed blue line),
since in the system with exponential switching, when inspection finds the system empty, the server switches to the slow rate and does not switch back till another inspection occurs.
When $\gamma=1000$, the system switches almost immediately and therefore the sojourn time distribution is very close to the one of the pure $M/M/1$ system.

\begin{figure}[h]%
\centering
\begin{minipage}{0.45\textwidth}
\includegraphics[width=\textwidth]{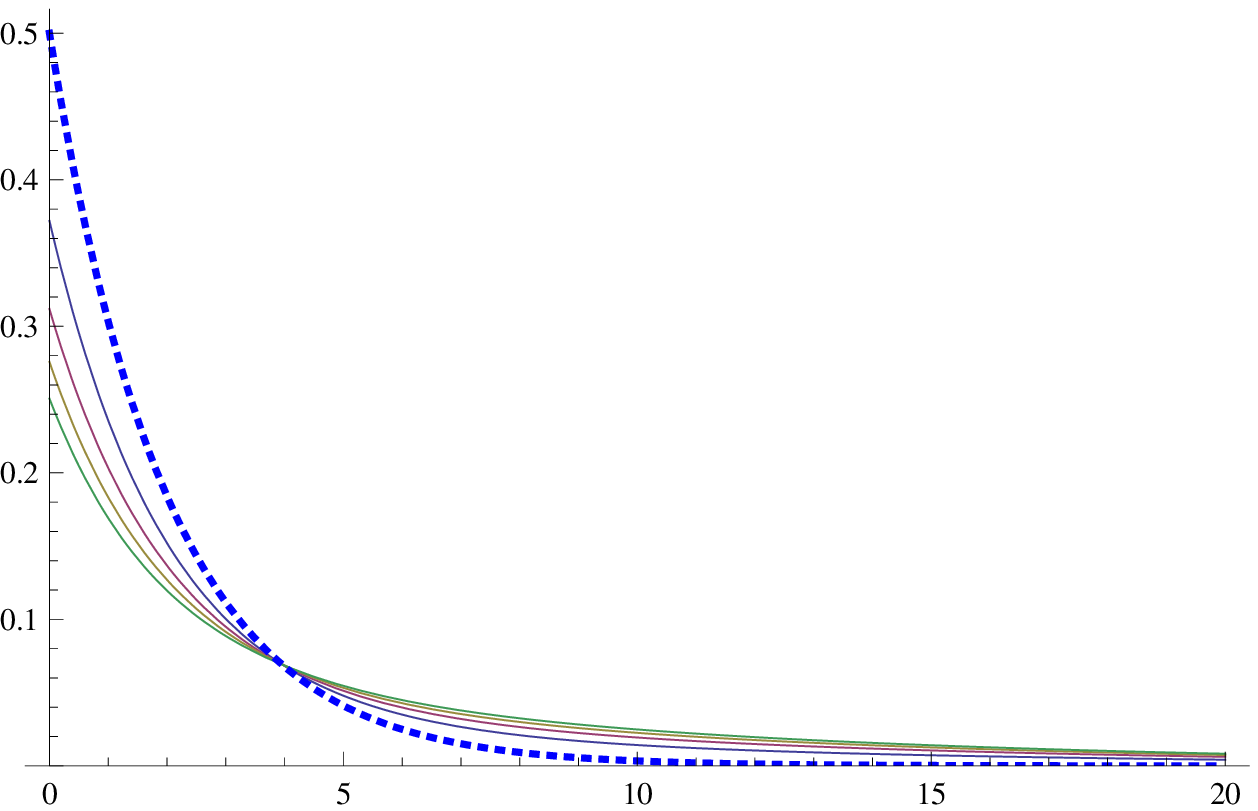}
  \caption{$\gamma = 1/100$}
  \label{fig:exp-sw-1/100}
\end{minipage}%
\qquad
\begin{minipage}{0.45\textwidth}
  \includegraphics[width=\textwidth]{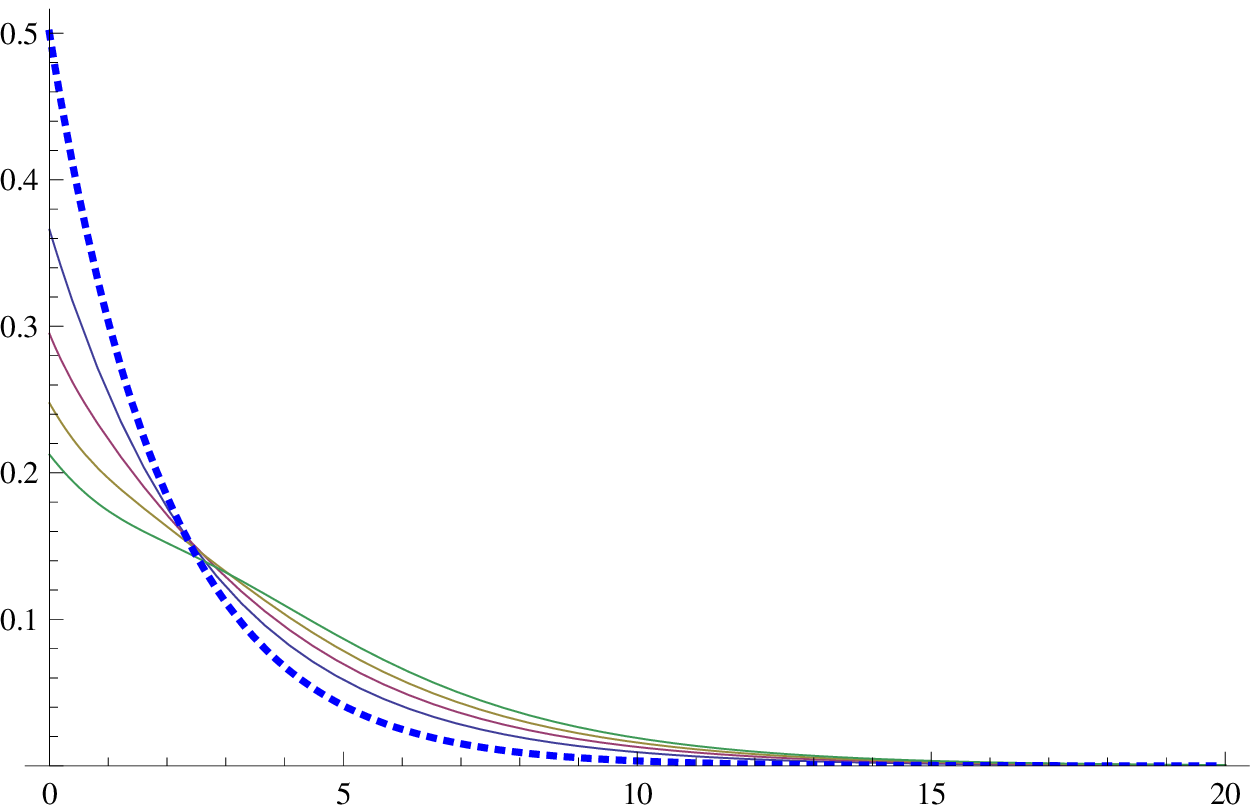}
  \caption{$\gamma = 1/10$}
  \label{fig:exp-sw-1/10}
\end{minipage}%
\\%
\begin{minipage}{0.45\textwidth}
\includegraphics[width=\textwidth]{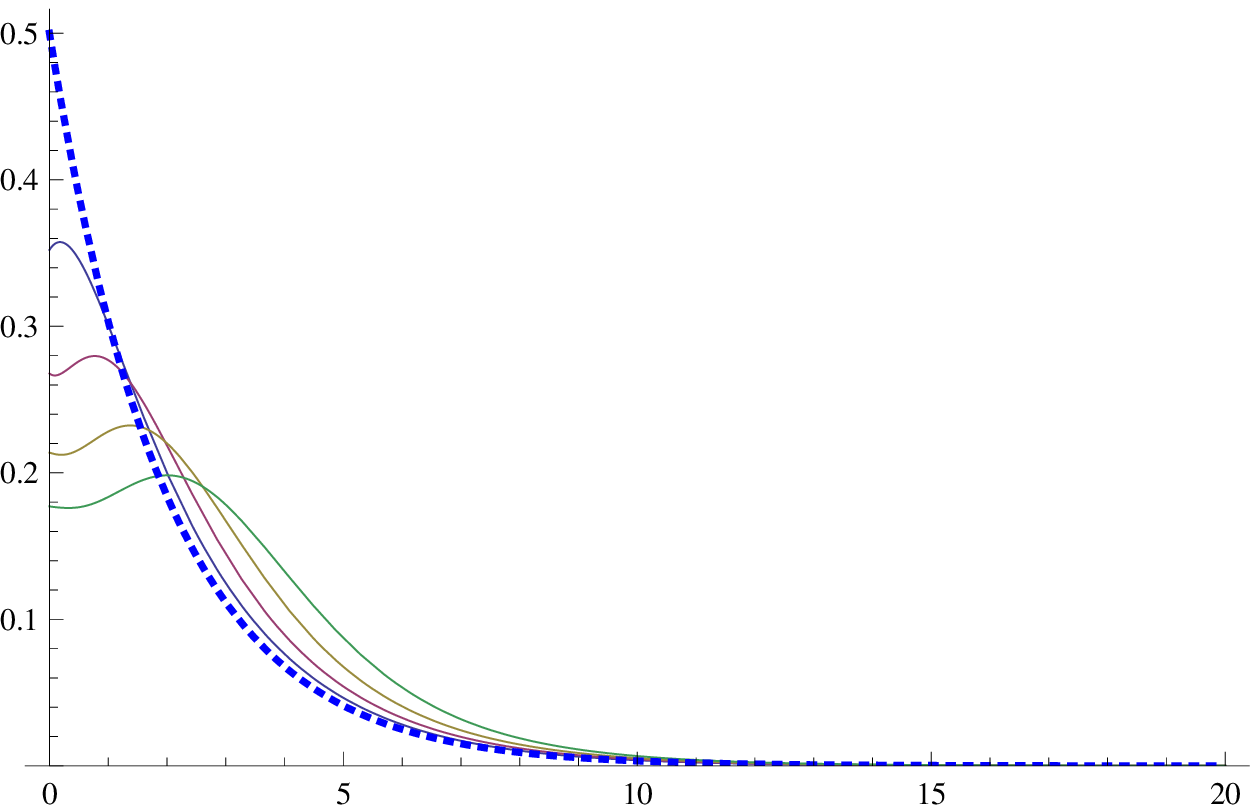}
  \caption{$\gamma = 1$}
  \label{fig:exp-sw-1}
\end{minipage}%
\qquad
\begin{minipage}{0.45\textwidth}
  \includegraphics[width=\textwidth]{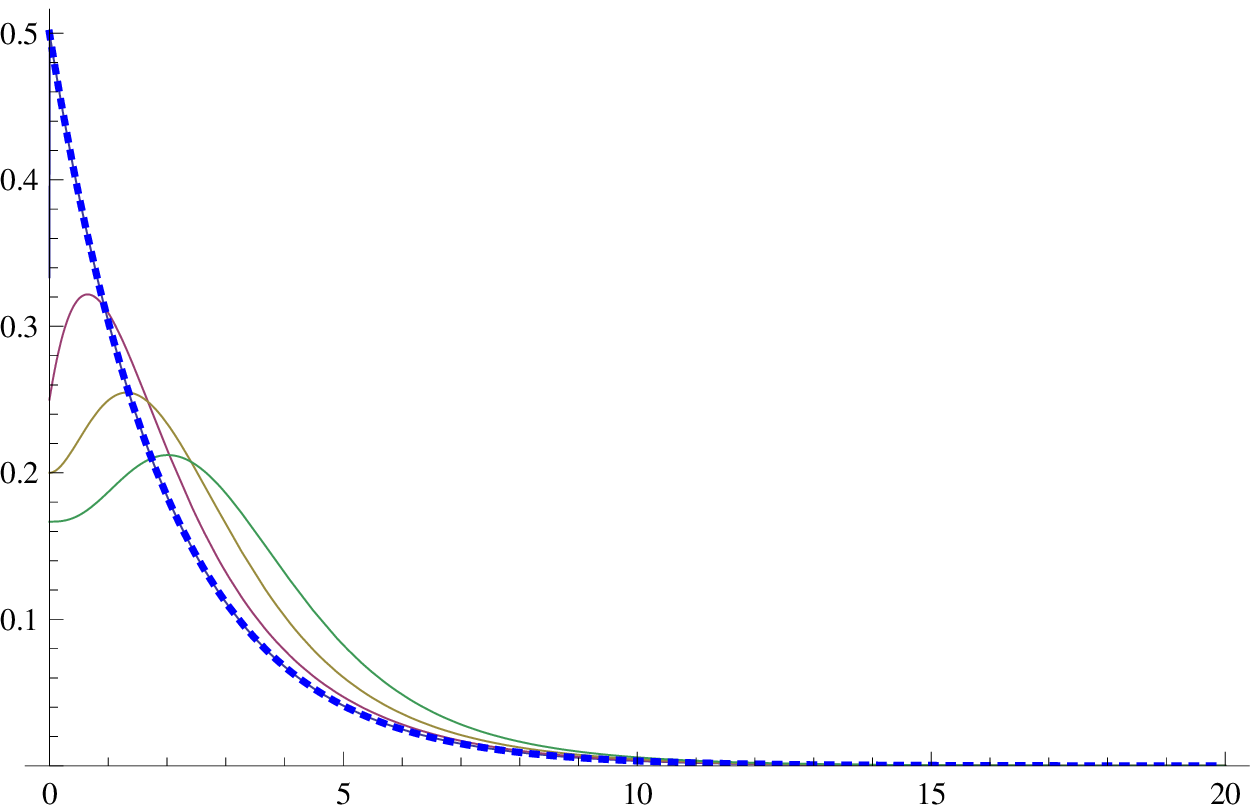}
  \caption{$\gamma = 1000$}
  \label{fig:exp-sw-1000}
\end{minipage}%

\vspace{.25cm}
Sojourn time distribution for: $\lambda = 1$, $\mu_0 = 1$, $\mu_1 = 3/2$ and
 $K\in\{%
{\color{blue} *},
{\color{curveONE} 0},
{\color{curveTWO} 1},
{\color{curveTHREE} 2},
{\color{curveFOUR} 3}
 \}$
 \label{fig:exponential-switching}
\end{figure}%

In Figures \ref{fig:exp-tp-imm-sw-1} and \ref{fig:exp-tp-imm-sw-3}, we plot again the results for $\lambda=1$, $\mu_0=1$ and $\mu_1=3/2$, but compare different values of $\gamma$'s.
One can see  that for $\gamma$ approaching $\lambda$, the system behaves very close to a system with immediate switching (shown as dashed black line).
Indeed, for values of $\gamma>1$, one cannot distinguish the curve from the limiting one.
This suggests that, checking the state of the system at a rate comparable to the arrival rate, can be considered from the point of view of the sojourn time as an immediate switching.
This could be used in the design phase of the system, when balancing between costs (by reducing service rate) and performance (by increasing the service and inspection rate).

\begin{figure}[h]
\centering
\begin{minipage}{0.45\textwidth}
\includegraphics[width=\textwidth]{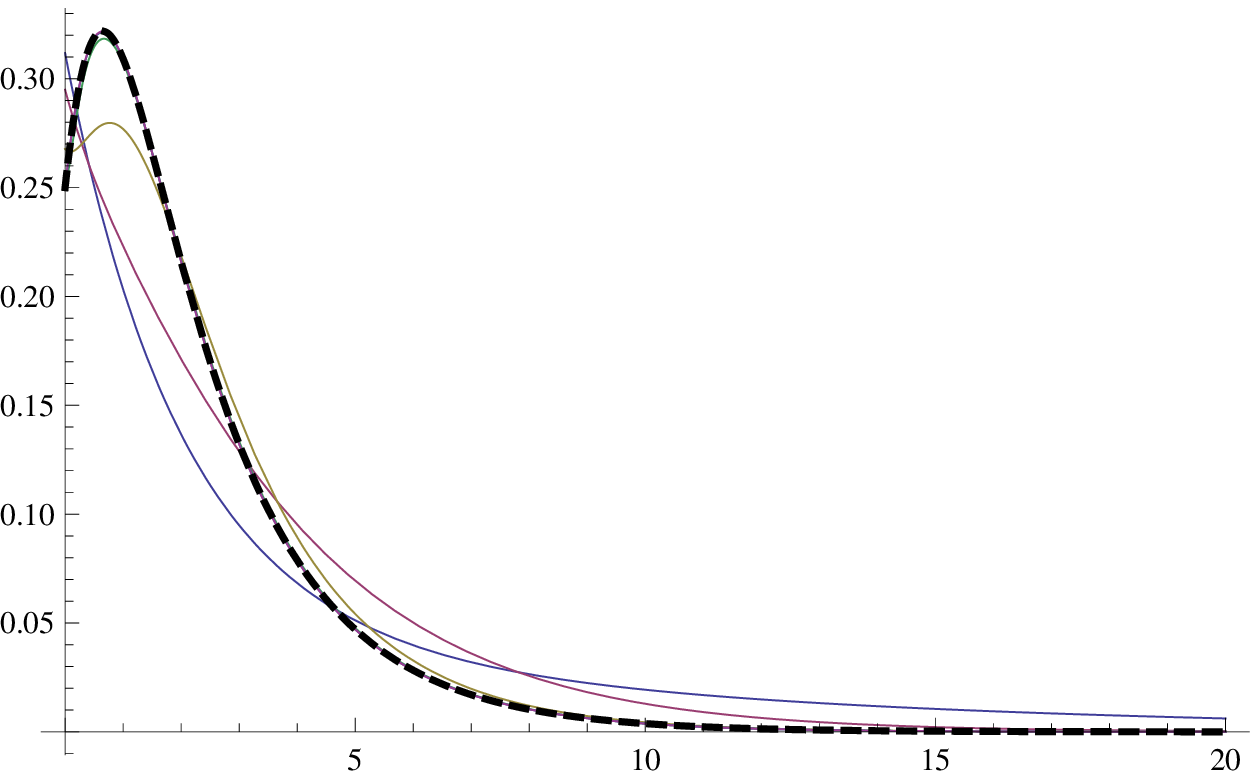}
  \caption{$K=1$}
  \label{fig:exp-tp-imm-sw-1}
\end{minipage}%
\qquad
\begin{minipage}{0.45\textwidth}
  \includegraphics[width=\textwidth]{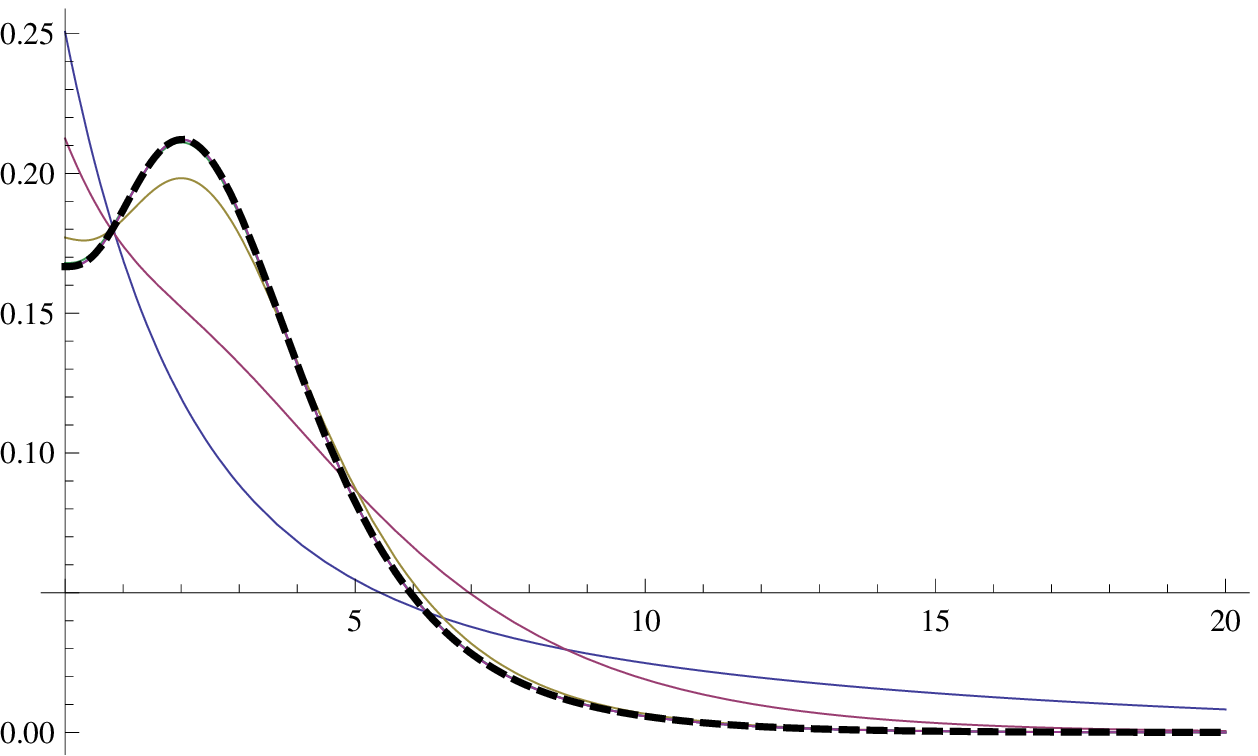}
  \caption{$K=3$}
  \label{fig:exp-tp-imm-sw-3}
\end{minipage}%

\vspace{.25cm}
Sojourn time distribution for: $\lambda = 1$, $\mu_0 = 1$, $\mu_1 = 3/2$ and
 $\gamma\in\{%
{\color{curveONE} 1/100},
{\color{curveTWO} 1/10},
{\color{curveTHREE} 1},
{\color{curveFOUR} 10},
{\color{curveFIVE} 100},
{\color{curveSIX} 1000},
{\color{black} \infty}
 \}$
 \label{fig:exponential-to-immediate-switching}
\end{figure}

%% file: conclusion.tex
\section{Conclusions}
\label{conclusions}
In this paper we studied the sojourn time distribution in an exponential single-server queueing system. Service is in order of arrival, and it is provided at low or high rate, which can be adapted at exponential inspection times, depending on the number of customers in the system.
To determine the Laplace transform of the stationary sojourn time distribution, we proposed a new methodological tool, that is \emph{matrix generating functions}. 
We used this tool to compute the Laplace transform of the sojourn time distribution in the system with inspection times.
Its expression is obtained recursively and shows a rational form that allows an immediate inverse-transformation. 
Numerical computations have shown, as expected,  that if the inspection rate is large, 
the sojourn time of the system with inspections converges to the one of the system with immediate switching.

We believe that the power of the matrix generating functions lies in its flexibility to analyze generalizations to phase-type services and inspection times.
An interesting and promising direction for future research is to explore the applicability of this tool to analyze the more general class of quasi-birth-and-death processes \cite{latouche}.

%% file: biblio.tex
\bibliographystyle{siam}      


%% file: appendix.tex
\appendix
\section{Technical proofs}\label{sec:appendix}
\begin{proof}[of Lemma \ref{lm:rec.psi}]
We can rewrite  the expression \eqref{eq:lap.rel} in the following form
\begin{equation}\label{eq:lap.rel.app}
\psi(n,m,s) = a_s(n+m) \, \psi(n-1,m,s) + b_s(n+m) \, \psi(n,m+1,s) \ .
\end{equation}
With $m=K-1$, equation \eqref{eq:rec.psi} becomes
\begin{align*}
\psi(n,K-1,s)
&= B_s(n+K-1,1) \,\psi(n,K,s) \\
 & \quad + \sum_{k=K-1}^{K-1} a_s(n+k) B_s(n+K-1,k-K+1) \, \psi(n-1,k,s) \\
&= b_s(n+K-1) \,\psi(n,K,s) + a_s(n+K-1) \, \psi(n-1,K-1,s)
\end{align*}
and therefore it holds true.
Now assume that equation \eqref{eq:rec.psi} is valid for $m+1$. Then by \eqref{eq:lap.rel.app},
\begin{align*}\label{eq:lap.rec.app}
\psi(n,m,s)
= & a_s(n+m) \, \psi(n-1,m,s)
    + b_s(n+m) \, B_s(n+m+1,K-1-m) \, \psi(n,K,s) \\
  & + b_s(n+m) \sum_{k=m+1}^{K-1} a_s(n+k) \, B_s(n+m+1,k-m-1) \, \psi(n-1,k,s)  \\
= & a_s(n+m) \, B_s(n+m,0)\, \psi(n-1,m,s)
    + B_s(n+m,K-m) \, \psi(n,K,s) \\
  & + \sum_{k=m+1}^{K-1} a_s(n+k) \, B_s(n+m,k-m) \, \psi(n-1,k,s)  \\
= & B_s(n+m,K-m) \, \psi(n,K,s)
    + \sum_{k=m}^{K-1} a_s(n+k) \, B_s(n+m,k-m) \, \psi(n-1,k,s)
\ .
\end{align*}
where we have used the fact that the definition of $B_s(k,h)$ implies that
$$b_s(k) \, B_s(k+1,h) = B_s(k,h+1) \ .$$
\end{proof}

\begin{proof}[of Lemma \ref{lm:S.function}]
By substituting in \eqref{eq:S.syst} the expression for $\matS$ given in \eqref{def:s} we get
\begin{align*}
\matB \, \matS \, \matZ = \sum_{h=0}^\infty \matB^{h+1} \, \matA \, \matZ^{h+1}  
= \sum_{h=0}^\infty  \matB^h \, \matA \, \matZ^h  \, - \matA = \matS - \matA
\end{align*}
which implies that the matrix $\matS$ is a solution of the matrix equation.

By assuming that $\matS$ and $\matS'$ are two solutions of this matrix equation, we would have that $\matY=\matS-\matS'$ is the solution of the following system
$$\matY = \matZ \, \matY \, \matB.$$
Iterating the last equation we get that
$$\matY = \matZ^n \, \matY \, \matB^n \quad n\geq0 \ .$$
This term converges to $0$ as $n\to\infty$ by the assumptions on the eigenvalues of the matrices $\matZ$ and $\matB$.
It follows that  $\matY=0$ and hence $\matS$ is unique.
\end{proof}

\begin{proof}[of Lemma \ref{lm:rel.S}]
The result follows from the following algebraic manipulations
\begin{align*}
&\sum_{h=0}^\infty \matS(\matT_2 \, \matT_1^h \,\matZ \, \matT_1^{-h} \, \matT_2^{-1}, \matA, \matB)
 \, \matT_2 \, \matT_1^h \, \matZ^h \\
& \quad = \sum_{h=0}^\infty 
\sum_{k=0}^\infty  \matB^k \, \matA \, (\matT_2 \, \matT_1^h \,\matZ \, \matT_1^{-h} \, \matT_2^{-1})^k  
\, \matT_2 \, \matT_1^h \, \matZ^h \\
& \quad = \sum_{h=0}^\infty 
\sum_{k=0}^\infty  \matB^k \, \matA \, \matT_2 \, \matT_1^h \,\matZ^k \, \matT_1^{-h} \, \matT_2^{-1}  
\, \matT_2 \, \matT_1^h \, \matZ^h \\
& \quad = \sum_{h=0}^\infty 
\sum_{k=0}^\infty  \matB^k \, \matA \, \matT_2 \, \matT_1^h \,\matZ^{k+h} 
= \sum_{k=0}^\infty \matB^k \, \matA \, \matT_2 \left(\sum_{h=0}^\infty \matT_1^h \,\matZ^h\right) \matZ^k \\
& \quad = \sum_{k=0}^\infty \matB^k \, \matA \, \matT_2 \, \matS(\matZ,I, \matT_1) \, \matZ^k 
= \matS(\matZ, \matA \, \matT_2 \, \matS(\matZ,I, \matT_1), \matB) 
\end{align*}
\end{proof}

\begin{lemma}\label{lm:S.der}
Let  $\matS(s) = \matS(\matZ,\matA(s),\matB(s))$, then its derivative in $s$ can be computed as the solution of the following linear system.
\begin{equation}\label{eq:der.S}
\matS'(s) - \matB(s) \, \matS'(s) \, \matZ - \matB'(s)\, \matS(s) \, \matZ = \matA'(s)
\end{equation}
\end{lemma}
\begin{proof}
By \eqref{eq:S.syst} we have that
\begin{align}
\matS(\matZ,\matA(s+h),\matB(s+h)) - \matB(s+h) \, \matS(\matZ,\matA(s+h),\matB(s+h)) \, \matZ &= \matA(s+h) \\
\matS(\matZ,\matA(s),\matB(s)) - \matB(s) \, \matS(\matZ,\matA(s),\matB(s)) \, \matZ &= \matA(s) \ .
\end{align}
Subtracting the expressions above, adding and removing $\matB(s+h) \, \matS(\matZ,\matA(s),\matB(s)) \, \matZ $ we have
\begin{equation*}
\Delta \matS(s) - \matB(s+h) \, \Delta \matS(s)   \, \matZ  - \Delta \matB(s) \, \matS(\matZ,\matA(s),\matB(s)) \, \matZ = \Delta \matA(s)
\end{equation*}
with $\Delta \matS(s) = \matS(\matZ,\matA(s+h),\matB(s+h)) -  \matS(\matZ,\matA(s),\matB(s)) $ and similar notations for $\Delta \matA(s)$ and $\Delta \matB(s)$.
Dividing for $h$ and letting $h\to0$ the result follows.
\end{proof}